\documentclass[a4paper, 12pt, reqno,final]{amsart}
\usepackage{amsaddr} 

\usepackage[utf8x]{inputenc}
\usepackage[T1]{fontenc}
\usepackage[english]{babel}
\usepackage{amsmath, amsfonts, amsthm, amssymb, amscd}
\usepackage{graphicx,subfig}
\usepackage[below]{placeins}
\usepackage{subfig}
\usepackage{multirow}
\usepackage{mathtools}
\usepackage{tikz}
\usepackage{pgfplots}
\usepackage{relsize}

\usepackage[hidelinks]{hyperref}
\hypersetup{
  pdfauthor = {Stefan Metzger},
  pdftitle = {Dynamic Boundary}
}

\usepackage[final]{showkeys} 
\renewcommand\showkeyslabelformat[1]{}
\usepackage{esint}
\usepackage{caption}
\usepackage{dsfont}

\usepackage{tikz}
\usetikzlibrary{patterns}
\usetikzlibrary{shapes.geometric}

\usepackage[disable]{todonotes}
\usepackage{diagbox}

\newtheorem{theorem}{Theorem}[section]
\newtheorem{lemma}[theorem]{Lemma}
\newtheorem*{lemma*}{Lemma}
\newtheorem{corollary}[theorem]{Corollary}

\newtheorem{remark}[theorem]{Remark}

\numberwithin{equation}{section}



\makeatletter
\newcommand{\labitem}[2]{%
\def\@itemlabel{\textbf{#1}}
\item
\def\@currentlabel{#1}\label{#2}}
\makeatother
\newcommand{\norm}[1]{\left\|{#1}\right\|}
\newcommand{\abs}[1]{\left|{#1}\right|}

\newcommand{\rkla}[1]{{\left(#1\right)}}
\newcommand{\trkla}[1]{{(#1)}}
\newcommand{\gkla}[1]{{\left\{#1\right\}}}
\newcommand{\tgkla}[1]{{\{#1\}}}
\newcommand{\skla}[1]{{\left\langle#1\right\rangle}}
\newcommand{\ekla}[1]{{\left[#1\right]}}
\newcommand{\tekla}[1]{{[#1]}}
\newcommand{\tabs}[1]{|{#1}|}

\newcommand{\bs}[1]{\boldsymbol{#1}}

\newcommand{\mb}[1]{\mathbf{#1}}

\newcommand{\iOmegaT}{{\int_{\Omega_T}}}

\newcommand{\iOmega}{\int_{\Omega}}

\newcommand{\iGamma}{\int_{\Gamma}}
\newcommand{\iGammaT}{\int_{\Gamma_T}}

\newcommand{\para}[1]{\partial _{#1}}
\newcommand{\dtau}{\para{\tau}^-}

\renewcommand{\div}{\operatorname{div}}

\newcommand{\diam}{\operatorname{diam}}

\newcommand{\eps}{\varepsilon}

\newcommand{\mol}[1][\empty]{
  \ifthenelse{\equal{#1}{\empty}}
  {\mathcal{J}_{\eps}}
  {\mathcal{J}_{\eps}\gkla{#1}}
}
\newcommand{\molb}[1][\empty]{
  \ifthenelse{\equal{#1}{\empty}}
  {\bs{\mathcal{J}}_{\eps}}
  {\bs{\mathcal{J}}_{\eps}\gkla{#1}}
}
\newcommand{\molbb}[1][\empty]{
  \ifthenelse{\equal{#1}{\empty}}
  {\bs{\mathfrak{J}}_{\eps}}
  {\bs{\mathfrak{J}}_{\eps}\gkla{#1}}
}

\newcommand{\molh}[1][\empty]{
  \ifthenelse{\equal{#1}{\empty}}
  {\mathcal{J}_{h}}
  {\mathcal{J}_{h}\gkla{#1}}
}
\newcommand{\molhb}[1][\empty]{
  \ifthenelse{\equal{#1}{\empty}}
  {\bs{\mathcal{J}}_{h}}
  {\bs{\mathcal{J}}_{h}\gkla{#1}}
}
\newcommand{\molhbb}[1][\empty]{
  \ifthenelse{\equal{#1}{\empty}}
  {\bs{\mathfrak{J}}_{h}}
  {\bs{\mathfrak{J}}_{h}\gkla{#1}}
}

\newcommand{\ids}{I}

\newcommand{\nn}{^{n}}

\newcommand{\no}{^{n-1}}

\newcommand{\n}{n}

\newcommand{\tl}{^{\tau}}
\newcommand{\tp}{^{\tau,+}}
\newcommand{\tm}{^{\tau,-}}
\newcommand{\tpm}{^{\tau,(\pm)}}

\newcommand{\h}{_{h}}

\newcommand{\weak}{\rightharpoonup}
\newcommand{\weakstar}{\stackrel{*}{\rightharpoonup}}

\newcommand{\Uhs}{U\h^\Omega}
\newcommand{\UhG}{U\h^\Gamma}

\newcommand{\Th}{\mathcal{T}_h}

\newcommand{\R}{\mathds{R}}

\newcommand{\projUhsop}{\mathcal{P}_{\Uhs}}

\newcommand{\Ihop}{\mathcal{I}_h}

\newcommand{\Ih}[1]{\Ihop\gkla{#1}}

\newcommand{\IhGop}{\mathcal{I}_h^\Gamma}
\newcommand{\IhG}[1]{\IhGop\gkla{#1}}

\newcommand{\nnorm}[1]{\left|\!\left|\!\left|#1\right|\!\right|\!\right|}

\newcommand{\restr}[2]{\ensuremath{
  \left.\kern-\nulldelimiterspace 
  #1 
  \vphantom{\big|} 
  \right|_{#2} 
  }}
  
\newcommand{\extend}[2]{\ensuremath{
  \left.\kern-\nulldelimiterspace 
  #1 
  \vphantom{\big|} 
  \right|^{#2} 
  }}
  
\newcommand{\mobO}{m}
\newcommand{\mobG}{m_{\Gamma}}
\newcommand{\muO}{\mu}
\newcommand{\muOh}{\mu_{h}}
\newcommand{\muG}{\mu_{\Gamma}}
\newcommand{\muGh}{\mu_{\Gamma,h}}
\newcommand{\LapBel}{\Delta_\Gamma}
\newcommand{\nablaG}{\nabla_{\Gamma}}

\newcommand{\MOmega}{\mathbf{M}_\Omega}
\newcommand{\MGamma}{\mathbf{M}_\Gamma}
\newcommand{\LOmega}{\mathbf{L}_\Omega}
\newcommand{\LGamma}{\mathbf{L}_\Gamma}

\newcommand{\muvec}{P}
\newcommand{\muGvec}{P_\Gamma}

\newcommand{\GammaSymb}{\Gamma}
\newcommand{\OmegaSymb}{\Omega}
\newcommand{\InnerSymb}{\stackrel{\circ}{\Omega}}
\newcommand{\ones}{\bs{1}}
\newcommand{\onesG}{\bs{1}_\Gamma}
\newcommand{\deltaG}{\delta_\Gamma}

\newcommand{\antikappa}{\beta}
\newcommand{\traceop}{\gamma}
\newcommand{\trace}[1]{\traceop\trkla{#1}}
\newcommand{\Xkappa}{X_\kappa}

\newcommand{\RHSO}{\operatorname{R}_{\stackrel{\circ}{\Omega}}}
\newcommand{\RHSG}{\operatorname{R}_{\Gamma}}

\begin{document}
\title[]{An efficient and convergent finite element scheme for Cahn--Hilliard equations with dynamic boundary conditions\footnote{This work was funded by the NSF through grant number NSF-DMS 1759536.}}
\date{\today}
\author[S.~Metzger]{Stefan Metzger}
\address{Department of Applied Mathematics, Illinois Institute of Technology, Chicago IL, 60616, USA}
\email{smetzger2@iit.edu}

\begin{abstract}
The Cahn--Hilliard equation is a widely used model that describes amongst others phase separation processes of binary mixtures or two-phase flows.
In the recent years, different types of boundary conditions for the Cahn--Hilliard equation were proposed and analyzed. 
In this publication, we are concerned with the numerical treatment of a recent model which introduces an additional Cahn--Hilliard type equation on the boundary as closure for the Cahn--Hilliard equation in the domain [C. Liu, H. Wu, Arch. Ration. Mech. An., 2019].
By identifying a mapping between the phase-field parameter and the chemical potential inside of the domain, we are able to postulate an efficient, unconditionally energy stable finite element scheme.
Furthermore, we establish the convergence of discrete solutions towards suitable weak solutions of the original model.
This serves also as an additional pathway to establish existence of weak solutions.
Furthermore, we present simulations underlining the practicality of the proposed scheme and investigate its experimental order of convergence.
\end{abstract}

%
%
\selectlanguage{english}

\maketitle

\section{Introduction}
Different approaches to model the hydrodynamics of fluid mixtures have been widely used in literature.
In addition to the conventional sharp interface models which consist of separate hydrodynamic systems for each component of the mixture, there are diffuse interface models.
In these models, the hyper-surface description of the fluid-fluid interface is replaced by a small transition region, where mixing of the macroscopically immiscible fluids is allowed.
This leads to a smooth transition between the pure phases.
In this manuscript, we analyze a numerical scheme for a diffuse interface model featuring dynamic boundary conditions, which was derived recently by C. Liu and H. Wu \cite{Liu2019}.
In an open domain $\Omega$ with boundary $\Gamma=\partial\Omega$ and outer normal vector $\bs{n}$, this model reads
\begin{subequations}\label{eq:model}
\begin{align}
\para{t}\phi&=\mobO\Delta\muO&&\text{in~}\Omega\times\trkla{0,T}\,,\label{eq:model:Omega:phi}\\
\muO&=-\delta\sigma\Delta\phi+\delta^{-1}\sigma F^\prime\trkla{\phi}&&\text{in~}\Omega\times\trkla{0,T}\,,\label{eq:model:Omega:mu}\\
\nabla\muO\cdot\bs{n}&=0&&\text{on~}\Gamma\times\trkla{0,T}\,,\label{eq:model:bc:mu}\\
\para{t}\phi&=\mobG\LapBel\muG&&\text{on~}\Gamma\times\trkla{0,T}\,,\label{eq:model:Gamma:phi}\\
\muG&=-\deltaG\kappa\LapBel\phi + \deltaG^{-1}G^\prime\trkla{\phi} +\delta\sigma\nabla\phi\cdot\bs{n}&&\text{on~}\Gamma\times\trkla{0,T}\,\label{eq:model:Gamma:mu}
\end{align}
\end{subequations}
with suitable initial conditions for the phase-field parameter $\phi$ in $\overline{\Omega}$.
Here, $\Delta_\Gamma$ denotes the Laplace--Beltrami operator, the positive parameters $\mobO$ and $\mobG$ are the mobility constants in the domain and on the boundary, the constant $\sigma>0$ is related to the surface tension, $\kappa\geq0$ describes the influence of surface diffusion, and $\delta$ and $\deltaG$ prescribe the width of the transition area in the domain and on the boundary.
The potentials $F$ and $G$ which govern the chemical potentials $\mu$ and $\muG$ will be discussed below in more detail.
In \eqref{eq:model}, it is assumed that $\phi$ is defined on $\overline{\Omega}$ and that its evolution is governed by a chemical potential defined on $\Omega$ and an additional one defined on $\Gamma=\partial\Omega$.
In contrast to other approaches (cf. \cite{Goldstein2011}), $\muO$ and $\muG$ are distinct quantities which are only coupled via the normal derivative of $\phi$.

In the recent years, various different boundary conditions for Cahn--Hilliard equations have been discussed.
The easiest form of a diffuse interface model for two-phase flow reads
\begin{subequations}\label{eq:standard:CH}
\begin{align}
\para{t}\phi&=\mobO\Delta \mu&&\text{in~}\Omega\times\trkla{0,T}\,,\label{eq:CH:phi}\\
\mu&=-\delta\sigma\Delta\phi +\delta^{-1}\sigma F^\prime\trkla{\phi}&&\text{in~}\Omega\times\trkla{0,T}\,,\label{eq:CH:mu}\\
\nabla\mu\cdot\bs{n}&=0&&\text{on~}\Gamma\times\trkla{0,T}\,,\label{eq:CH:bc:mu}\\
\nabla\phi\cdot\bs{n}&=0&&\text{on~}\Gamma\times\trkla{0,T}\,,\label{eq:CH:bc:phi}
\end{align}
\end{subequations}
in combination with initial conditions for $\phi$.
The chemical potential $\mu$ is given as the first variation of the free energy
\begin{align}
\mathcal{E}_{\Omega}\trkla{\phi}:= \sigma\delta\iOmega\tfrac12\abs{\nabla\phi}^2+\sigma\delta^{-1}\iOmega F\trkla{\phi}\,,
\end{align}
where $F$ is a double-well potential with minima in $\phi=\pm1$ representing the pure fluid phases.
Typical choices for $F$ are the logarithmic double-well potential \begin{align}
W_{\operatorname{log}}\trkla{\phi}:=\tfrac\theta2\trkla{1+\phi}\log\trkla{1+\phi} +\tfrac\theta2\trkla{1-\phi}\log\trkla{1-\phi} -\tfrac{\theta_c}2\phi^2
\end{align}
with $0<\theta<\theta_c$, the double obstacle potential
\begin{align}
W_{\operatorname{obst}}\trkla{\phi}:=\left\{\begin{matrix} \theta\trkla{1-\phi^2}&\phi\in\tekla{-1,+1}\\\infty&\text{else}\end{matrix}\right.&&\text{with~}\theta>0\,,
\end{align}
and the polynomial double-well potential $W_{\operatorname{pol}}\trkla{\phi}:=\tfrac14\trkla{\phi^2-1}^2$. 
Cahn--Hilliard equations with polynomial double-well potentials are investigated, e.g.~in \cite{Elliott86,Zheng1986,Pego89,BatesFife93,Grinfeld1995,RybkaHoffmann99}.
For Cahn--Hilliard equations with the singular potentials $W_{\operatorname{log}}$ and $W_{\operatorname{obst}}$, we refer the reader to \cite{Blowey1991,Elliott1996,AbelsWilke2007,Cherfils2011}.

The boundary condition \eqref{eq:CH:bc:mu} states that there is no flux across $\Gamma$, i.e. $\iOmega\phi$ is conserved.
The second boundary condition \eqref{eq:CH:bc:phi} indicates that the fluid-fluid interface, i.e. the zero level set of $\phi$, intersects the boundary $\Gamma$ at a static contact angle of $\tfrac\pi2$.
This can be interpreted as neglecting the interactions between the fluids and the walls of the surrounding container.
Although \eqref{eq:standard:CH} satisfies the energy balance equation
\begin{align}\label{eq:E1}
\restr{\mathcal{E}_{\Omega}\trkla{\phi}}{T}+\int_0^T\!\!\!\!\iOmega\mobO\abs{\nabla\mu}^2 =\restr{\mathcal{E}_{\Omega}\trkla{\phi}}{0}\,,
\end{align}
the boundary condition \eqref{eq:CH:bc:phi} imposing a static contact angle is considered a major flaw and there are several attempts to improve this boundary condition.
For an improved description of the occurring boundary effects, the surface energy functional 
\begin{align}
\mathcal{E}_{\Gamma}\trkla{\phi}:=\kappa\deltaG\iGamma\tfrac12\abs{\nablaG\phi}^2 +\deltaG^{-1}\iGamma G\trkla{\phi}\,
\end{align}
was introduced with $\kappa\geq0$ and $\deltaG>0$.
Here, $\nablaG$ denotes the surface gradient operator on $\Gamma$ and $G$ denotes a suitable boundary potential which might be chosen similar to $F$.
Dynamic boundary conditions were picked in a way that the total energy $\tekla{\mathcal{E}_{\Omega}+\mathcal{E}_{\Gamma}}$ is decreasing in time.
An example for such boundary conditions are Allen--Cahn-type boundary conditions (cf. \cite{Fischer1997, Fischer1998, RackeZheng2003, Wu2004, ChillFasangovaPruess2006, Gilardi2009, Miranville2010, Cherfils2011, Liero2013, Colli2014, Colli2015, Mininni2017}), where \eqref{eq:CH:bc:phi} is replaced by
\begin{subequations}\label{eq:bc:AC}
\begin{align}
\para{t}\phi&=-\mobG\muG&&\text{on~}\Gamma\times\trkla{0,T}\,,\\
\muG&=-\kappa\deltaG\Delta_\Gamma\phi+\deltaG^{-1}G^\prime\trkla{\phi}+\delta\sigma\nabla\phi\cdot\bs{n}&&\text{on~}\Gamma\times\trkla{0,T}\,.
\end{align}
\end{subequations}
In the case $\kappa=0$ and $\deltaG=1$, this boundary conditions reduces to
\begin{align}\label{eq:dynamic_angle}
\sigma\delta\nabla\phi\cdot\bs{n}=-\tfrac1{\mobG}\para{t}\phi-G^\prime\trkla{\phi}\,,
\end{align}
where the potential $G$ interpolates between the liquid-solid interfacial energies of the two fluid phases and prescribes the stationary contact angle via Young's formula. 
This boundary condition was used in \cite{Qian2006} to describe dynamic contact angles (see also \cite{Thompson1989}).
Models combining the boundary conditions \eqref{eq:CH:bc:mu} and \eqref{eq:bc:AC} satisfy an energy equality of the form
\begin{align}
\restr{\mathcal{E}_{\Omega}\trkla{\phi}}{T}+\restr{\mathcal{E}_{\Gamma}\trkla{\phi}}{T}+\int_0^T\ekla{\iOmega\mobO\abs{\nabla\mu}^2 + \int_{\partial\Omega}\tfrac1{\mobG}\abs{\para{t}\phi}^2}=\restr{\mathcal{E}_{\Omega}\trkla{\phi}}{0}+\restr{\mathcal{E}_{\Gamma}\trkla{\phi}}{0}\,
\end{align}
and conserve the mean value of $\phi$ in $\Omega$, i.e. we have $\restr{\iOmega\phi}{T}=\restr{\iOmega\phi}{0}$.

In \cite{Gal2006} a coupled boundary condition replacing \eqref{eq:CH:bc:mu} and \eqref{eq:CH:bc:phi} was proposed.
Assuming that $\muG$ is the trace of $\muO$, boundary conditions of the form
\begin{subequations}\label{eq:Gal}
\begin{align}
\para{t}\phi&=-\mobG\muO-\mobO\nabla\muO\cdot\bs{n}&&\text{on~}\Gamma\times\trkla{0,T}\,,\\
\muO&=-\kappa\deltaG\Delta_\Gamma\phi+\deltaG^{-1}G^\prime\trkla{\phi}+\sigma\delta\nabla\phi\cdot\bs{n}&&\text{on~}\Gamma\times\trkla{0,T}
\end{align}
\end{subequations}
were introduced.
A solution to \eqref{eq:CH:phi}, \eqref{eq:CH:mu}, and \eqref{eq:Gal} still minimizes $\mathcal{E}_\Omega+\mathcal{E}_\Gamma$ in the sense that
\begin{align}
\restr{\mathcal{E}_{\Omega}\trkla{\phi}}{T}+\restr{\mathcal{E}_{\Gamma}\trkla{\phi}}{T}+\int_0^T\ekla{\iOmega\mobO\abs{\nabla\mu}^2 + \int_{\partial\Omega}\mobG\abs{\muO}^2}=\restr{\mathcal{E}_{\Omega}\trkla{\phi}}{0}+\restr{\mathcal{E}_{\Gamma}\trkla{\phi}}{0}\,
\end{align}
holds true.
However, the boundary conditions \eqref{eq:Gal} do not allow for conservation of $\phi$.

A third approach was proposed by Goldstein et al.~in \cite{Goldstein2011}.
In this publication, it was also assumed that $\muG$ is the trace of $\muO$.
However, \eqref{eq:Gal} was replaced by a Cahn--Hilliard-type boundary equations of the form
\begin{subequations}\label{eq:Goldstein}
\begin{align}
\para{t}\phi&=\mobG\Delta_\Gamma\muO-\mobO\nabla\muO\cdot\bs{n}&&\text{on~}\Gamma\times\trkla{0,T}\,,\\
\muO&=-\kappa\deltaG\Delta_\Gamma\phi+\deltaG^{-1}G^\prime\trkla{\phi}+\sigma\delta\nabla\phi\cdot\bs{n}&&\text{on~}\Gamma\times\trkla{0,T}\,.
\end{align}
\end{subequations}
Solutions to \eqref{eq:CH:phi}, \eqref{eq:CH:mu}, and \eqref{eq:Goldstein} satisfy the energy equality
\begin{align}\label{eq:energy:CH}
\restr{\mathcal{E}_{\Omega}\trkla{\phi}}{T}+\restr{\mathcal{E}_{\Gamma}\trkla{\phi}}{T}+\int_0^T\ekla{\iOmega\mobO\abs{\nabla\mu}^2 + \int_{\partial\Omega}\mobG\abs{\nablaG\muO}^2}=\restr{\mathcal{E}_{\Omega}\trkla{\phi}}{0}+\restr{\mathcal{E}_{\Gamma}\trkla{\phi}}{0}\,.
\end{align}
Furthermore, the total mass of $\phi$ is conserved, i.e.~$\para{t}\tekla{\iOmega\phi+\iGamma\phi}=0$.
A similar model was also discussed in \cite{Motoda2018}.

A fourth approach was proposed by C. Liu and H. Wu.
In \cite{Liu2019}, they derived model \eqref{eq:model} via a variational approach using different flow maps for $\Omega$ and $\Gamma$.
This model also satisfies \eqref{eq:energy:CH}.
However, the underlying assumptions on the occurring boundary effects are different.
While the boundary conditions proposed in \cite{Goldstein2011} allow for mass transfer between $\Omega$ and $\Gamma$ and enforce an instant equilibration of the chemical potentials, i.e. $\muG$ has to be the trace of $\muO$, the model derived in \cite{Liu2019} allows for differences in the chemical potentials, but prohibits mass transfer, i.e. $\iOmega\phi$ and $\iGamma\phi$ are conserved individually.
For boundary conditions interpolating between \eqref{eq:Goldstein} and \eqref{eq:model:bc:mu}-\eqref{eq:model:Gamma:mu}, which can be interpreted as non-instantaneous adsorption processes, we refer the reader to \cite{KnopfLamLiuMetzger_arxiv_2020}.

A first existence and uniqueness result for weak and strong solutions to \eqref{eq:model} was provided in \cite{Liu2019} by constructing solutions to a regularized system, where \eqref{eq:model:Omega:mu} and \eqref{eq:model:Gamma:mu} are extended by $\alpha\para{t}\phi$, and taking the limit $\alpha\searrow0$.\\
A different pathway to proving the existence of solutions to \eqref{eq:model} was used by Garcke and Knopf in \cite{GarckeKnopf20}.
They interpreted model \eqref{eq:model} as a gradient flow equation to the total free energy $\mathcal{E}_\Omega\trkla{\phi}+\mathcal{E}_\Gamma\trkla{\phi}$ and used this structure for their proof of existence and uniqueness of weak solutions.
A comparable wellposedness result derived from a fully discrete finite element scheme can be found in Theorem \ref{th:convergence} of this manuscript.\\
The numerical treatment of the Cahn--Hilliard equation and its variants -- often in combination with Navier--Stokes-equations -- was intensely discussed through the last years.
Consequently, there are various different discretization techniques at hand, which transfer the energy stability \eqref{eq:E1} to a discrete setting.
These techniques include approaches based on convex-concave splittings of the energy (cf. \cite{WiseWangLowengrub2009,ShenWangWangWise2012}) or the polynomial double-well potential (cf.~\cite{Elliott1993} and  \cite{Grun2013c,Grun2013,GarckeHinzeKahle2016, GrunGuillenMetzger2016} for an application of Cahn--Hilliard--Navier--Stokes-systems), stabilized linearly implicit approaches (cf. \cite{XuTang2006,ShenYang2010}), the method of invariant energy quadratization (cf. \cite{ChengYangShen2017,YangZhang2017arxiv}) and the recently developed scalar auxiliary variable approach (see \cite{LiShen2019arxiv}).\\
Although, we will restrict ourselves to non-singular potentials, we do not want to conceal that there are also numerical schemes at hand which are able to deal with the singular potentials $W_{\operatorname{log}}$ and $W_{\operatorname{obst}}$ (see e.g.~\cite{CopettiElliott92,Blowey1992, Blowey1996,Barrett1999,Barrett2001, Frank2020}).\\

In this publication, we are interested in the numerical treatment of \eqref{eq:model}.
A finite difference model for the treatment of the Allen--Cahn-type boundary conditions \eqref{eq:bc:AC} was proposed in \cite{Kenzler2001}.
A first finite element scheme for model \eqref{eq:model} was proposed in the Bachelor's thesis \cite{Trautwein2018} (see also \cite{GarckeKnopf20} for numerical results).
In this thesis, a straightforward, fully implicit discretization based on continuous, piecewise linear finite element functions
was applied to model \eqref{eq:model}, and the arising nonlinear system was solved using Newton's method.
In this publication, we pursue a different approach and investigate the connection between $\phi$ and the chemical potentials.\\
The peculiarity of \eqref{eq:model} is the coupling between the chemical potential $\muO$ defined inside of the domain and the $\muG$ on the boundary.
In the standard Cahn--Hilliard equation \eqref{eq:standard:CH}, the chemical potential is merely a definition in terms of $\phi$.
This allows us to write \eqref{eq:standard:CH} as a sole, nonlinear, fourth-order equation (see e.g. \cite{Grun2013c,GrunGuillenMetzger2016,GarckeHinzeKahle2016}).
In \eqref{eq:model}, however, the chemical potentials $\muO$ and $\muG$ are coupled via the normal derivative $\nabla\phi\cdot\bs{n}$.
Therefore, its weak form formally reads
\begin{subequations}\label{eq:weak}
\begin{align}
\int_0^T\!\!\!\!\iOmega\para{t}\phi\theta &+\mobO\int_0^T\!\!\!\!\iOmega\nabla\muO\cdot\nabla\theta=0\,,\label{eq:weak:Omega}\\
\int_0^T\!\!\!\!\iGamma\para{t}\phi\tilde{\theta}+&\mobG\int_0^T\!\!\!\!\iGamma\nablaG\muG\cdot\nablaG\tilde{\theta}=0\,,\label{eq:weak:Gamma}\\
\begin{split}
\int_0^T\!\!\!\!\iOmega\muO\hat{\theta}+\int_0^T\!\!\!\!\iGamma\muG\hat{\theta}=&\delta\sigma\int_0^T\!\!\!\!\iOmega\nabla\phi\cdot\nabla\hat{\theta} +\delta^{-1}\sigma\int_0^T\!\!\!\!\iOmega F^\prime\trkla{\phi}\hat{\theta} \\
&+ \kappa\deltaG\int_0^T\!\!\!\!\iGamma\nablaG\phi\cdot\nablaG\hat{\theta} +\deltaG^{-1}\int_0^T\!\!\!\!\iGamma G^\prime\trkla{\phi}\hat{\theta}\label{eq:weak:mus}
\end{split}
\end{align}
\end{subequations}
with sufficiently regular $\theta$, $\tilde{\theta}$, and $\hat{\theta}$.
In particular, we have only one equation for $\muO$ and $\muG$.
Consequently, the chemical potentials have to be determined by solving a system consisting of \eqref{eq:weak:mus} and the additional assumption that $\phi$ is continuous on $\overline{\Omega}$.
The latter one translates to the constraint that \eqref{eq:weak:Omega} and \eqref{eq:weak:Gamma} yield compatible results.
Deducing a suitable expression for $\muO$ will be key ingredient for the derivation of an efficient numerical scheme, but also for the numerical analysis, as the existence of a unique (discrete) $\muO$ for any given $\phi$ allows us to reuse techniques from the analysis of the standard Cahn--Hilliard equations.
As we will discuss in Remark \ref{rem:efficiency}, this approach also prevents the arising linear system from degenerating for vanishing time increments.

The outline of the paper is as follows. 
In Section \ref{sec:scheme}, we introduce the discrete function spaces and derive the discrete scheme.
In Section \ref{sec:analysis}, we will establish a first a priori estimate which is discrete counterpart of \eqref{eq:energy:CH}, and use this estimate to prove the existence of discrete solutions.
The main convergence result, Theorem \ref{th:convergence}, which also provides the existence of weak solutions, can be found in Section \ref{sec:convergence}, where we establish improved regularity results and show the convergence of discrete solutions towards weak solutions of \eqref{eq:model}.
For uniqueness results for these weak solutions, we refer the reader to Section 5 in \cite{GarckeKnopf20}.
We will conclude this section by briefly discussing the case of Allen--Cahn-type boundary conditions (cf. Remark \ref{rem:ac}).
By showing that the presented techniques are also applicable for Allen--Cahn-type boundary conditions, we also cover \eqref{eq:bc:AC} and its special case \eqref{eq:dynamic_angle} suggested in \cite{Qian2006}.
In Section \ref{sec:simulation}, we present numerical simulations of phase-separation processes to underline the practicality of the scheme.
We shall also validate our scheme in terms of mass conservation, energy dissipation, and compatibility of \eqref{eq:weak:Omega} and \eqref{eq:weak:Gamma}.

\paragraph{Notation}
Given the spatial domain $\Omega\subset\R^d$ with $d\in\tgkla{2,3}$ and a time interval $\trkla{0,T}$, we denote the
space-time cylinder $\Omega\times\trkla{0,T}$ by $\Omega_T$.
By $W^{k,p}\trkla{\Omega}$ we denote the space of $k$-times weakly differentiable functions with weak derivatives in $L^p\trkla{\Omega}$.
The symbol $W^{k,p}_0\trkla{\Omega}$ stands for the closure of $C^\infty_0\trkla{\Omega}$ in $W^{k,p}\trkla{\Omega}$.
For $p = 2$, we will denote $W^{k,2}\trkla{\Omega}$ by $H^k\trkla{\Omega}$ and $W^{k,2}_0\trkla{\Omega}$ by $H^k_0\trkla{\Omega}$.
The dual space of $H^1\trkla{\Omega}$ will be denoted by $\trkla{H^1\trkla{\Omega}}^\prime$ and the corresponding dual pairing by $\skla{.,.}$.
For a Banach space $X$ and a time interval $I$, the symbol $L^p\trkla{I; X}$ stands for the
parabolic space of $L^p$-integrable functions on $I$ with values in $X$.
We use a notation similar to the one introduced above for function spaces defined on $\Gamma$.
In this publication, we are concerned with domains $\Omega$ having a lipschitzian boundary $\Gamma$.
In this case, the spaces $L^p\trkla{\Gamma}$, $W^{1,p}\trkla{\Gamma}$, and $H^1\trkla{\Gamma}$ are well-defined (cf.~\cite{Kufner77}).
We denote the dual pairing between $\trkla{H^1\trkla{\Gamma}}^\prime$ and $H^1\trkla{\Gamma}$ by $\skla{.,.}_\Gamma$.
In addition, we define the space 
\begin{align}
\Xkappa:=\left\{\begin{matrix}
\tgkla{v\in H^1\trkla{\Omega}\,:\,\trace{v}\in H^1\trkla{\Gamma}} & \text{if~}\kappa>1\,,\\
H^1\trkla{\Omega}&\text{if~}\kappa=0\,,
\end{matrix}\right.
\end{align}
where $\traceop$ defines the trace operator.
For domains with lipschitzian boundaries, the trace operator is uniquely defined and lies in $\mathcal{L}\trkla{W^{1,p}\trkla{\Omega},W^{1-1/p,p}\trkla{\Gamma}}$ (cf.~\cite{Necas2012}).
For brevity, we will sometimes (in particular when the considered function is continuous) neglect the trace operator and write $v$ instead of $\trace{v}$.
\section{Derivation of an efficient numerical scheme}\label{sec:scheme}

We start by introducing the general notation and the discretization techniques used in the considered scheme.
Concerning the discretization with respect to time, we assume that
\begin{itemize}
\labitem{(T)}{item:disc:time} the time interval $I:=[0,T)$ is subdivided in intervals $I_n:=[t_n,t_{n+1})$ with $t_{n+1}=t_n+\tau_n$ for time increments $\tau_n>0$ and $n=0,...,N-1$ with $t_N=T$. For simplicity, we take $\tau_n\equiv\tau=\tfrac{T}{N}$ for $n=0,...,N-1$.
\end{itemize}
The spatial domain $\Omega\subset\R^{d}$ in spatial dimensions $d\in\tgkla{2,3}$ is assumed to be bounded and convex.
To avoid additional technicalities, we will asssume that $\Omega$ is polygonal (or polyhedral, respectively).
We introduce partitions $\Th$ of $\Omega$ and $\Th^\Gamma$ of $\Gamma$ depending on a spatial discretization parameter $h > 0$ satisfying the following assumptions:
\begin{itemize}
\labitem{(S1)}{item:disc:space} Let $\tgkla{\Th}_{h>0}$ be a quasiuniform family (in the sense of \cite{BrennerScott}) of partitions of $\Omega$ into disjoint, open simplices $K$, so that
\begin{align*}
\overline{\Omega}\equiv\bigcup_{K\in\Th}\overline{K}&&&\text{with }\max_{K\in\Th}\diam\trkla{K}\leq h\,.
\end{align*}
\labitem{(S2)}{item:disc:gamma} Let $\tgkla{\Th^\Gamma}_{h>0}$ be a quasiuniform family of partitions of $\Gamma$ into disjoint, open simplices $K^\Gamma$, so that
\begin{align*}
\forall K^\Gamma\in\Th^\Gamma~~\exists!K\in\Th\text{~such~that~} \overline{K^\Gamma}=\overline{K}\cap\Gamma\,,
\end{align*}
and
\begin{align*}
\Gamma\equiv\bigcup_{K^\Gamma\in\Th^\Gamma}\overline{K^\Gamma}&&&\text{with }\max_{K^\Gamma\in\Th^\Gamma}\diam\trkla{K^\Gamma}\leq h\,.
\end{align*}
\end{itemize}
\ref{item:disc:gamma} implies that $\Th^\Gamma$ is compatible to $\Th$ in the sense that all elements in $\Th^\Gamma$ are edges (or faces) of elements in $\Th$.
For the approximation of the phase-field $\phi$ and the chemical potential $\muO$ we use continuous, piecewise linear finite element functions on $\Th$.
This space will be denoted by $\Uhs$ and is spanned by the functions $\tgkla{\chi_{h,k}}_{k=1,...,\dim\Uhs}$ forming a dual basis to the vertices $\tgkla{\bs{x}_k}_{k=1,...,\dim\Uhs}$ of $\Th$, i.e. $\chi_{h,k}\trkla{\bs{x}_k}=\delta_{k,l}$ for $k,l=1,...,\dim\Uhs$.
Analogously, we denote the space of continuous, piecewise linear finite element functions on $\Th^\Gamma$ by $\UhG$.
This space is spanned by functions $\tgkla {\chi_{h,k}^\Gamma}_{k=1,...,\dim\UhG}$ forming a dual basis to the vertices $\tgkla{\bs{x}_k^\Gamma}_{k=1,...,\dim\UhG}$ of $\Th^\Gamma$, i.e. $\chi^\Gamma_{h,k}\trkla{\bs{x}^\Gamma_k}=\delta_{k,l}$ for $k,l=1,...,\dim\Uhs$.
Due to the compatibility condition for $\Th$ and $\Th^\Gamma$, we have 
\begin{align}\label{eq:compatibility}
\UhG=\operatorname{span}\tgkla{\trace{\theta\h}\,:\,\theta\h\in\Uhs}\,.
\end{align}
Without loss of generality, we may assume that the first $\dim\UhG$ vertices of $\Th$ are located on $\Gamma$, i.e. $\tgkla{\bs{x}_k^\Gamma}_{k=1,...,\dim\UhG}=\tgkla{\bs{x}_k}_{k=1,...,\dim\UhG}$.
We define the nodal interpolation operators $\Ihop\,:\,C^0\trkla{\overline{\Omega}}\rightarrow \Uhs$ and $\IhGop\,:\,C^0\trkla{\overline{\Gamma}}\rightarrow \UhG$ by
\begin{align}\label{eq:def:Ih}
\Ih{a}:=\sum_{k=1}^{\dim\Uhs}a\trkla{\bs{x}_k}\chi_{h,k}\,,&&\text{and}&&
\IhG{a}:=\sum_{k=1}^{\dim\UhG}a\trkla{\bs{x}_k}\chi^\Gamma_{h,k}\,.
\end{align}
For future reference, we state the following estimate for the interpolation operators.
\begin{lemma}\label{lem:ihfe}
Let $\Th$ and $\Th^\Gamma$ satisfy \ref{item:disc:space} and \ref{item:disc:gamma}.
Furthermore, let $p\in[1,\infty)$, $1\leq q\leq\infty$, and $q^*=\tfrac{q-1}q$ for $q<\infty$ or $q^*=1$ for $q=\infty$.
Then
\begin{align}
\norm{\trkla{\ids-\Ihop}\tgkla{f\h g\h}}_{L^p\trkla{\Omega}}&\leq C h^2\norm{\nabla f\h}_{L^{pq}\trkla{\Omega}}\norm{\nabla g\h}_{L^{pq^*}\trkla{\Omega}}\,,\\
\norm{\trkla{\ids-\IhGop}\tgkla{f\h g\h}}_{L^p\trkla{\Gamma}}&\leq C h^2\norm{\nablaG f\h}_{L^{pq}\trkla{\Gamma}}\norm{\nablaG g\h}_{L^{pq^*}\trkla{\Gamma}}\,.
\end{align}
holds true for all $f\h,\,g\h\in\Uhs$.
\end{lemma}
\begin{proof}
Using the standard error estimates for the nodal interpolation operator (cf. \cite{BrennerScott}), we compute on each $K\in\Th$:
\begin{align}\label{eq:Ih:estimate1}
\int_{K}\abs{\trkla{\ids-\Ihop}\tgkla{f\h g\h}}^p\leq C h^{2p}\int_K\abs{f\h g\h}_{W^{2,\infty}\trkla{K}}^p\,.
\end{align}
As $f\h,\,g\h\in\Uhs$, they are linear on each $K$, i.e.~their second spatial derivatives vanish.
Therefore, we obtain
\begin{align}\label{eq:Ih:estimate2}
\begin{split}
\tabs{f\h g\h}_{W^{2,\infty}\trkla{K}}=&\,\max_{i,j=1,...,d}\norm{\partial_i\partial_j\trkla{f\h g\h}}_{L^\infty\trkla{K}}=\max_{i,j=1,...,d}\norm{\partial_if\h\partial_jg\h}_{L^\infty\trkla{K}}\\
\leq&\,\max_{i,j=1,...,d}\norm{\partial_if\h}_{L^\infty\trkla{K}}\norm{\partial_jg\h}_{L^\infty\trkla{K}}
\end{split}
\end{align}
As the first spatial derivatives of $f\h$ and $g\h$ are constant, we combine \eqref{eq:Ih:estimate1} and \eqref{eq:Ih:estimate2} and apply Hölder's inequality to obtain
\begin{align}
\begin{split}
\int_{K}\abs{\trkla{\ids-\Ihop}\tgkla{f\h g\h}}^p \leq Ch^{2p}\sum_{i,j=1}^d\int_K\tabs{\partial_i f\h}^p\tabs{\partial_j g\h}^p\\
\leq Ch^{2p}\norm{\nabla f\h}_{L^{pq}\trkla{K}}^p \norm{\nabla g\h}_{L^{pq^*}\trkla{K}}^p\,.
\end{split}
\end{align}
Similar computations provide the result for $\IhGop$.
\end{proof}
Concerning the potentials $F$ and $G$, we make the following assumptions:
\begin{itemize}
\labitem{(P1)}{item:potentials} $F,G\in C^1\trkla{\R}$ are bounded from below, i.e. there exists a constant $C>0$ such that $F\trkla{s}>-C$ and $G\trkla{s}>-C$ for all $s\in\R$.
Furthermore, there exist convex, non-negative functions $F_+,\,G_+\in C^1\trkla{\R}$ and concave functions $F_-,\,G_-\in C^1\trkla{\R}$ such that $F\equiv F_++F_-$ and $G\equiv G_++G_-$.
\labitem{(P2)}{item:potentialsbounds} The convex and concave parts of $F$ and $G$ can be further decomposed into a polynomial part of degree four and an additional part with a globally Lipschitz-continuous first derivative.
Moreover, there exists $\antikappa\geq0$ such that the concave parts satisfy
\begin{align*}
G_-^\prime\trkla{s_2}\trkla{s_1-s_2}\geq G_-\trkla{s_1}-G_-\trkla{s_2}+\antikappa\abs{s_1-s_2}^2
\end{align*}
for all $s_1,\,s_2\in\R$. In the case $\kappa=0$, we assume that the above assumption holds true for $\antikappa>0$.
\end{itemize}
\begin{remark}
The Assumptions \ref{item:potentials} and \ref{item:potentialsbounds} are in particular satisfied by the polynomial double-well potential $W_{\operatorname{pol}}\trkla{\phi}:=\tfrac14\trkla{1-\phi^2}^2$ and the penalized double-well potential
\begin{align}\label{eq:penpot}
W_{\operatorname{pen}}\trkla{\phi}:=W_{\operatorname{pol}}\trkla{\phi}+C_{\operatorname{pen}}\max\tgkla{\trkla{\abs{\phi}-1},0}^2&&\text{with~} C_{\operatorname{pen}}>0\,.
\end{align}
The latter one is often used in practical computations, as it penalizes $\phi\notin\tekla{-1,+1}$ but does not introduce singularities (cf. \cite{GrunGuillenMetzger2016}).\\
The logarithmic potential $W_{\operatorname{log}}$ and the double obstacle potential $W_{\operatorname{obst}}$, which
have the advantage of restricting $\phi$ to the interval $\tekla{-1,+1}$, do not satisfy \ref{item:potentials} and \ref{item:potentialsbounds} and are therefore not considered in this manuscript.
\end{remark}
\begin{remark}
In this publication, we consider only a convex-concave decomposition of the double-well potential.
Other suitable, energy stable discretization techniques can be found in \cite{GrunGuillenMetzger2016}.
For a comparison of these techniques, we refer to \cite{Metzger2018}.
\end{remark}
Using the notation introduced above and the compatibility condition \eqref{eq:compatibility}, we may write our discrete scheme as follows:
For given $\phi\h\no\in\Uhs$, find $\trkla{\phi\h\nn,\,\muOh\nn\,,\muGh\nn}\in\Uhs\times\Uhs\times\UhG$ satisfying
\begin{subequations}\label{eq:feform}
\begin{align}
\iOmega\Ih{\phi\h\nn \theta\h} + \tau\mobO\iOmega\nabla\muOh\nn\cdot\nabla\theta\h =& \iOmega\Ih{\phi\h\no\theta\h}\,,\label{eq:feform:Omega:phi}\\
\iGamma\IhG{\phi\h\nn \theta\h} + \tau\mobG\iGamma\nablaG\muGh\nn\cdot\nablaG\theta\h =& \iGamma\IhG{\phi\h\no\theta\h}\,,\label{eq:feform:Gamma:phi}
\end{align}
\begin{multline}
\iOmega\Ih{\muOh\nn\theta\h} + \iGamma\IhG{\muGh\nn\theta\h} = \delta\sigma\iOmega\nabla\phi\h\nn\cdot\nabla\theta\h \\
+\delta^{-1}\sigma\iOmega \Ih{\rkla{F_+^\prime\trkla{\phi\h\nn}+F_-^\prime\trkla{\phi\h\no}}\theta\h} \\
+\kappa\deltaG\iGamma\nablaG\phi\h\nn\cdot\nablaG\theta\h +\deltaG^{-1}\iGamma \IhG{\rkla{G_+^\prime\trkla{\phi\h\nn}+G_-^\prime\trkla{\phi\h\no}}\theta\h} \label{eq:feform:mu}
\end{multline}
\end{subequations}
for all $\theta\h\in\Uhs$.
As discussed on the example of the weak formulation \eqref{eq:weak}, \eqref{eq:feform} only provides on equation for both chemical potentials.
Consequently, the goal for this section will be to derive an equivalent formulation for \eqref{eq:feform} with unique expressions for $\muOh\nn$ and $\muGh\nn$, which allows us to reuse the standard techniques established for \eqref{eq:standard:CH}.\\
We define the lumped mass matrices $\MOmega$ and $\MGamma$ via
\begin{subequations}\label{eq:def:matrix}
\begin{align}
\trkla{\MOmega}_{ij} &:=\iOmega\Ih{\chi_{hj}\chi_{hi}}&&\forall i,j=1,...,\dim\Uhs\,,\\
\trkla{\MGamma}_{ij} &:=\iGamma\IhG{\chi^\Gamma_{hj}\chi^\Gamma_{hi}} && \forall i,j=1,...,\dim\UhG\,,
\end{align}
and the stiffness matrices $\LOmega$ and $\LGamma$ via
\begin{align}
\trkla{\LOmega}_{ij} &:=\iOmega\nabla\chi_{hj}\cdot\nabla\chi_{hi}&&\forall i,j=1,...,\dim\Uhs\,,\\
\trkla{\LGamma}_{ij} &:=\iGamma\nablaG\chi^\Gamma_{hj}\cdot\nablaG\chi^\Gamma_{hi} && \forall i,j=1,...,\dim\UhG\,.
\end{align}
\end{subequations}
Furthermore, we collect the nodal values of $\phi\h\nn$, $\phi\h\no$, $\muOh\nn$, and $\muGh\nn$ in the vectors $\Phi\nn$, $\Phi\no$, $\muvec\nn$, and $\muGvec\nn$.
In a slight misuse of notation, we will write $F\trkla{\Phi\nn}$, when we apply a function $F$ to all components of $\Phi\nn$.
With this notation, we are able to rewrite \eqref{eq:feform} as
\begin{subequations}\label{eq:matrixform}
\begin{align}
\MOmega\Phi\nn +\tau\mobO\LOmega\muvec\nn=&\MOmega\Phi\no\,,\label{eq:matrixform:Omega:phi}\\
\MGamma\trkla{\restr{\Phi\nn}{\GammaSymb}} +\tau\mobG\LGamma\muGvec\nn=&\MGamma\trkla{\restr{\Phi\no}{\GammaSymb}}\,,\label{eq:matrixform:Gamma:phi}
\end{align}
\begin{multline}\label{eq:matrixform:mu}
\MOmega\muvec\nn + \extend{\rkla{\MGamma\muGvec\nn}}{\OmegaSymb}=\delta\sigma\LOmega\Phi\nn+\delta^{-1}\sigma\MOmega F_+^\prime\trkla{\Phi\nn} +\delta^{-1}\sigma\MOmega F_-^\prime\trkla{\Phi\no}\\
+\extend{\rkla{\deltaG\kappa\LGamma\trkla{\restr{\Phi\nn}{\GammaSymb}} +\deltaG^{-1}\MGamma G^\prime_+\trkla{\restr{\Phi\nn}{\GammaSymb}} +\deltaG^{-1}\MGamma G^\prime_-\trkla{\restr{\Phi\no}{\GammaSymb}}}} {\OmegaSymb}\,.
\end{multline}
\end{subequations}
Here, we used the extension operator $\extend{.}{\OmegaSymb}\,:\, \R^{\dim\UhG}\rightarrow \R^{\dim\Uhs}$ defined via
\begin{align}
\R^{\dim\UhG}\ni A\mapsto \begin{psmallmatrix}A\\ 0\end{psmallmatrix}\in\R^{\dim\Uhs}
\end{align}
and the restriction operator $\restr{.}{\GammaSymb}\,:\,\R^{\dim\Uhs}\rightarrow\R^{\dim\UhG}$, which restricts a vector its first $\dim\UhG$ entries.
\\
In order to derive a scheme allowing to solve \eqref{eq:matrixform}, we define restriction operators for matrices.
In particular, we will split a matrix $\mb{A}\in\R^{\dim\Uhs\times\dim\Uhs}$ into submatrices
\begin{align}\label{eq:def:restriction}
\begin{matrix*}[l]
&\restr{\mb{A}}{\GammaSymb\times\GammaSymb}\in\R^{\dim\UhG\times\dim\UhG}\,,& \restr{\mb{A}}{\GammaSymb\times\InnerSymb}\in\R^{\dim\UhG\times\trkla{\dim\Uhs-\dim\UhG}}\,,\\
& \restr{\mb{A}}{\InnerSymb\times\GammaSymb}\in\R^{\trkla{\dim\Uhs-\dim\UhG}\times\dim\UhG}\,,&\restr{\mb{A}}{\InnerSymb\times\InnerSymb}\in\R^{\trkla{\dim\Uhs-\dim\UhG}\times\trkla{\dim\Uhs-\dim\UhG}}\,,\\
& \restr{\mb{A}}{\GammaSymb\times\OmegaSymb}\in\R^{\dim\UhG\times\dim\Uhs}\,, &\restr{\mb{A}}{\InnerSymb\times\OmegaSymb}\in\R^{\trkla{\dim\Uhs-\dim\UhG}\times\dim\Uhs}\,,\\
&\restr{\mb{A}}{\OmegaSymb\times\GammaSymb}\in\R^{\dim\Uhs\times\dim\UhG}\,,\quad\text{and }&\restr{\mb{A}}{\OmegaSymb\times\InnerSymb}\in\R^{\dim\Uhs\times\trkla{\dim\Uhs-\dim\UhG}}\,,
\end{matrix*}
\end{align}
such that
\begin{align}
\mb{A}=\begin{pmatrix}
\restr{\mb{A}}{\GammaSymb\times\GammaSymb} & \restr{\mb{A}}{\GammaSymb\times\InnerSymb}\\
\restr{\mb{A}}{\InnerSymb\times\GammaSymb} & \restr{\mb{A}}{\InnerSymb\times\InnerSymb}
\end{pmatrix} = \begin{pmatrix}
\restr{\mb{A}}{\GammaSymb\times\OmegaSymb}\\\restr{\mb{A}}{\InnerSymb\times\OmegaSymb}
\end{pmatrix} = \begin{pmatrix}
\restr{\mb{A}}{\OmegaSymb\times\GammaSymb}&\restr{\mb{A}}{\OmegaSymb\times\InnerSymb}
\end{pmatrix}\,.
\end{align}
Hence, the chemical potentials are given as solutions of the $\trkla{\dim\Uhs+\dim\UhG}\times\trkla{\dim\Uhs+\dim\UhG}$-system
\begin{align}\label{eq:splitMatrix}
\begin{pmatrix}
\restr{\MOmega}{\GammaSymb\times\GammaSymb} \!& \mb{0} &\! \MGamma\\
\mb{0} \!& \restr{\MOmega}{\InnerSymb\times\InnerSymb} &\! \mb{0}\\
\mobO\restr{\MOmega^{-1}}{\GammaSymb\times\OmegaSymb}\restr{\LOmega}{\OmegaSymb\times\GammaSymb} \!& \mobO\restr{\MOmega^{-1}}{\GammaSymb\times\OmegaSymb}\restr{\LOmega}{\OmegaSymb\times\InnerSymb} &\! -\mobG\MGamma^{-1}\LGamma 
\end{pmatrix}\!\!\!
\begin{pmatrix}
\restr{\muvec\nn}{\GammaSymb}\\
\restr{\muvec\nn}{\InnerSymb}\\
\muGvec\nn
\end{pmatrix}\!
=\!\begin{pmatrix}
\RHSG\trkla{\Phi\nn}\\
\RHSO\trkla{\Phi\nn}\\
0
\end{pmatrix}
\end{align}
\begin{subequations}\label{eq:rhs}
\begin{align}
\begin{split}
\text{with~}\RHSG\trkla{\Phi\nn}:=& \delta\sigma\restr{\LOmega}{\GammaSymb\times\OmegaSymb}\Phi\nn + \delta^{-1}\sigma\restr{\MOmega}{\GammaSymb\times\OmegaSymb}F_+^\prime\trkla{\Phi\nn} +\delta^{-1}\sigma \restr{\MOmega}{\GammaSymb\times\OmegaSymb}F_-^\prime\trkla{\Phi\no}\\
&+\kappa\deltaG\LGamma\restr{\Phi\nn}{\GammaSymb} + \deltaG^{-1}\MGamma G^\prime_+\trkla{\restr{\Phi\nn}{\GammaSymb}}+ \deltaG^{-1}\MGamma G^\prime_-\trkla{\restr{\Phi\no}{\GammaSymb}}
\end{split}\\
\text{and~} \RHSO\trkla{\Phi\nn}:=& \delta\sigma\restr{\LOmega}{\InnerSymb\times\OmegaSymb}\Phi\nn + \delta^{-1}\sigma\restr{\MOmega}{\InnerSymb\times\OmegaSymb}F_+^\prime\trkla{\Phi\nn} +\delta^{-1}\sigma \restr{\MOmega}{\InnerSymb\times\OmegaSymb}F_-^\prime\trkla{\Phi\no}\,.
\end{align}
\end{subequations}
Here, the first two lines are a consequence of \eqref{eq:matrixform:mu} and the last line guarantees that \eqref{eq:matrixform:Omega:phi} and \eqref{eq:matrixform:Gamma:phi} provide the same result for $\restr{\Phi\nn}{\GammaSymb}$.\\
As the \eqref{eq:matrixform} is nonlinear in $\Phi\nn$, computing a possible solution requires the application of an iterative scheme (e.g. Newton's method) and therefore solving \eqref{eq:splitMatrix} multiple times per time step.
Hence, solving a $\trkla{\dim\Uhs+\dim\UhG} \times\trkla{\dim\Uhs+\dim\UhG}$-system each time is not desirable and we have to continue reducing the complexity of the system.

From the second line in \eqref{eq:splitMatrix}, we immediately get
\begin{align}\label{eq:def:muInner}
\restr{\muvec\nn}{\InnerSymb}= \restr{\MOmega}{\InnerSymb\times\InnerSymb}^{-1}\RHSO\trkla{\Phi\nn}\,,
\end{align}
while the first line provides 
\begin{align}\label{eq:def:muG}
\muGvec\nn = -\MGamma^{-1} \restr{\MOmega}{\GammaSymb\times\GammaSymb} \restr{\muvec\nn}{\Gamma} +\MGamma^{-1} \RHSG\trkla{\Phi\nn}\,.
\end{align}
Using \eqref{eq:def:muInner} and \eqref{eq:def:muG}, we may write the last line in \eqref{eq:splitMatrix} as
\begin{align}
\begin{split}
\mobO \restr{\MOmega^{-1}}{\GammaSymb\times\OmegaSymb}\restr{\LOmega}{\OmegaSymb\times\GammaSymb} \restr{\muvec\nn}{\GammaSymb}=&-\mobO\restr{\MOmega^{-1}}{\GammaSymb\times\OmegaSymb}\restr{\LOmega}{\OmegaSymb\times\InnerSymb}\restr{\muvec\nn}{\InnerSymb}+\mobG\MGamma^{-1}\LGamma \muGvec\nn\\
=&-\mobO\restr{\MOmega^{-1}}{\GammaSymb\times\OmegaSymb}\restr{\LOmega}{\OmegaSymb\times\InnerSymb}\restr{\MOmega}{\InnerSymb\times\InnerSymb}^{-1}\RHSO\trkla{\Phi\nn}\\
&-\mobG\MGamma^{-1}\LGamma\MGamma^{-1} \restr{\MOmega}{\GammaSymb\times\GammaSymb} \restr{\muvec\nn}{\Gamma}\\
&+\mobG\MGamma^{-1}\LGamma\MGamma^{-1} \RHSG\trkla{\Phi\nn}\,,
\end{split}
\end{align}
and therefore
\begin{multline}\label{eq:mu:step1}
\rkla{\mobO \restr{\MOmega^{-1}}{\GammaSymb\times\OmegaSymb}\restr{\LOmega}{\OmegaSymb\times\GammaSymb} +\mobG\MGamma^{-1}\LGamma\MGamma^{-1} \restr{\MOmega}{\GammaSymb\times\GammaSymb}}\restr{\muvec\nn}{\GammaSymb}\\
 = -\mobO\restr{\MOmega^{-1}}{\GammaSymb\times\OmegaSymb}\restr{\LOmega}{\OmegaSymb\times\InnerSymb}\restr{\MOmega}{\InnerSymb\times\InnerSymb}^{-1}\RHSO\trkla{\Phi\nn}+\mobG\MGamma^{-1}\LGamma\MGamma^{-1} \RHSG\trkla{\Phi\nn}\,.
\end{multline}
As $\MOmega^{-1}$ is a diagonal matrix, $\restr{\MOmega^{-1}}{\GammaSymb\times\OmegaSymb}\restr{\LOmega}{\OmegaSymb\times\GammaSymb}=\restr{\MOmega^{-1}}{\GammaSymb\times\GammaSymb}\restr{\LOmega}{\GammaSymb\times\GammaSymb}$ holds true.
This allows us to multiply \eqref{eq:mu:step1} by $\restr{\MOmega}{\GammaSymb\times\GammaSymb}$ to obtain
\begin{multline}\label{eq:mu:step2}
\rkla{\mobO \restr{\LOmega}{\GammaSymb\times\GammaSymb} +\mobG\restr{\MOmega}{\GammaSymb\times\GammaSymb}\MGamma^{-1}\LGamma\MGamma^{-1} \restr{\MOmega}{\GammaSymb\times\GammaSymb}}\restr{\muvec\nn}{\GammaSymb}\\
 = -\mobO\restr{\LOmega}{\GammaSymb\times\InnerSymb}\restr{\MOmega}{\InnerSymb\times\InnerSymb}^{-1}\RHSO\trkla{\Phi\nn}+\mobG\restr{\MOmega}{\GammaSymb\times\GammaSymb}\MGamma^{-1}\LGamma\MGamma^{-1} \RHSG\trkla{\Phi\nn}\,.
\end{multline}
In order to show that \eqref{eq:mu:step2} provides a well-defined expression for $\restr{\muvec\nn}{\GammaSymb}$, we need to prove that the matrix on the left-hand side is indeed invertible.
\begin{lemma}
The matrix $\rkla{\mobO \restr{\LOmega}{\GammaSymb\times\GammaSymb} +\mobG\restr{\MOmega}{\GammaSymb\times\GammaSymb}\MGamma^{-1}\LGamma\MGamma^{-1} \restr{\MOmega}{\GammaSymb\times\GammaSymb}}$, that is defined via \eqref{eq:def:matrix} and \eqref{eq:def:restriction}, is symmetric, positive definite.
\end{lemma}
\begin{proof}
It is obvious that $\mobO \restr{\LOmega}{\GammaSymb\times\GammaSymb}$ and $\mobG\restr{\MOmega}{\GammaSymb\times\GammaSymb}\MGamma^{-1}\LGamma\MGamma^{-1} \restr{\MOmega}{\GammaSymb\times\GammaSymb}$ are symmetric, positive semi-definite matrices.
Therefore, it will be sufficient to show that $A^T\restr{\LOmega}{\GammaSymb\times\GammaSymb}A>0$ for all $0\neq A\in\R^{\dim\UhG}$ to complete the proof.
This is equivalent to showing 
\begin{align}\label{eq:spd}
\tilde{A}^T\LOmega\tilde{A}>0\text{~~with~~}\tilde{A}=\extend{A}{\Omega}=\begin{psmallmatrix}A\\0\end{psmallmatrix}\text{~~for~all~~}0\neq A\in\R^{\dim\UhG}\,.
\end{align}
From \eqref{eq:def:matrix}, we have that $\LOmega$ is symmetric, positive semi-definite with only constant vectors corresponding to the zero eigenvalue.
As the restrictions in \eqref{eq:spd} do not allow for constant vectors, the proof is complete. 
\end{proof}
Combining \eqref{eq:mu:step2} with \eqref{eq:def:muInner}, we obtain an expression for the chemical potential which requires us to solve only a $\dim\UhG$ by $\dim\UhG$ linear system with a sparse, symmetric, positive definite matrix.
Having an expression for the chemical potential, we propose the following nonlinear equation for $\Phi\nn$.
For given $\Phi\no\in\R^{\dim\Uhs}$, we compute $\Phi\nn\in\R^{\dim\Uhs}$ satisfying
\begin{multline}\label{eq:scheme}
\Phi\nn +\tau\mobO\MOmega^{-1}\LOmega \begin{pmatrix}
\rkla{\mobO \restr{\LOmega}{\GammaSymb\times\GammaSymb} +\mobG\restr{\MOmega}{\GammaSymb\times\GammaSymb}\MGamma^{-1}\LGamma\MGamma^{-1} \restr{\MOmega}{\GammaSymb\times\GammaSymb}}^{-1} & \mb{0}\\ \mb{0}&\mathds{1}
\end{pmatrix}\\
\cdot\begin{pmatrix}
-\mobO\restr{\LOmega}{\GammaSymb\times\InnerSymb}\restr{\MOmega}{\InnerSymb\times\InnerSymb}^{-1}\RHSO\trkla{\Phi\nn}+\mobG\restr{\MOmega}{\GammaSymb\times\GammaSymb}\MGamma^{-1}\LGamma\MGamma^{-1} \RHSG\trkla{\Phi\nn}\\
\restr{\MOmega}{\InnerSymb\times\InnerSymb}^{-1}\RHSO\trkla{\Phi\nn}
\end{pmatrix}
=\Phi\no\,.
\end{multline}
Here, $\RHSG$ and $\RHSO$, which are defined in \eqref{eq:rhs}, also depend on the known values $\Phi\no$.
As we will show in the next section, solutions $\Phi\nn$ to \eqref{eq:scheme} satisfy the compatibility condition used in \eqref{eq:splitMatrix}, which allows us to recover \eqref{eq:feform}.
\begin{remark}\label{rem:efficiency}
At this point, we want to discuss the advantages of formulation \eqref{eq:scheme}.
Although \eqref{eq:feform} can be written in a symmetric form, trying to solve \eqref{eq:feform} directly has one major flaw.
As there is no explicit formula for the chemical potentials available, one has to solve for $\phi\h\nn$, $\muOh\nn$, and $\muGh\nn$ monolithically.
However, this system degenerates for $\tau\searrow0$, i.e.~if we opt for a small time increment to capture rapid changes, we end up with an ill-conditioned system:
For $\tau\searrow0$, \eqref{eq:feform:Omega:phi} and \eqref{eq:feform:Gamma:phi} reduce to $\phi\h\nn=\phi\h\no$, i.e.~the dependencies on the chemical potentials vanishe leaving only \eqref{eq:feform:mu} to determine both potentials.\\
The proposed scheme \eqref{eq:scheme}, on the other hand, is based on another $\tau$-independent relation between the chemical potentials, which prevents the system from becoming ill-conditioned for vanishing time increments.
For an illustration of the dependence of the condition numbers on the size of the time increment based on practical computations, we refer the reader to Section \ref{sec:simulation}.\\
The downside of \eqref{eq:scheme} is that it requires us to solve an additional smaller linear system with the matrix
 $\rkla{\mobO \restr{\LOmega}{\GammaSymb\times\GammaSymb} +\mobG\restr{\MOmega}{\GammaSymb\times\GammaSymb}\MGamma^{-1}\LGamma\MGamma^{-1} \restr{\MOmega}{\GammaSymb\times\GammaSymb}}$ repeatedly.
However, as this matrix is symmetric and positive definite, it can be tackled efficiently using a conjugate gradient method.
\end{remark}
\section{Stability and existence of discrete solutions}\label{sec:analysis}
In this section, we analyze the discrete scheme \eqref{eq:scheme} proposed in the previous section.
As we derived explicit expressions for $\muOh\nn$ and $\muGh\nn$ in the previous section, we could return to the variational form \eqref{eq:feform} and derive stability and existence results from there.
Namely, testing \eqref{eq:feform:Omega:phi} by $\muOh\nn$, \eqref{eq:feform:Gamma:phi} by $\muGh\nn$, and \eqref{eq:feform:mu} by $\trkla{\phi\h\nn-\phi\h\no}$ will provide a discrete version of \eqref{eq:energy:CH}.
However, as \eqref{eq:scheme} is the formulation we suggest to implement, we establish first stability and existence results based on this formulation. By doing so, we shall verify that all information from \eqref{eq:feform} are preserved in \eqref{eq:scheme} and shed light on the structure of \eqref{eq:scheme}.\\
Although \eqref{eq:scheme} is entirely written in terms of the unknown quantity $\Phi\nn$, we will continue using $\muvec\nn$ and $\muGvec\nn$, which are defined in \eqref{eq:mu:step2}, \eqref{eq:def:muInner}, and \eqref{eq:def:muG}, to simplify the notation.
For the ease of representation, we will set $\sigma=\delta=\deltaG=1$ for the remainder of this publication.
As a first step, we shall verify that \eqref{eq:scheme} indeed satisfies the compatibility constraint $\mobO\restr{\MOmega^{-1}}{\GammaSymb\times\GammaSymb}\restr{\LOmega}{\GammaSymb\times\OmegaSymb}\muvec\nn=\mobG\MGamma^{-1}\LGamma\muGvec\nn$.
This auxiliary result allows us derive an a priori stability result for \eqref{eq:scheme} which serves as the corner stone for proving the existence of discrete solutions without additional restrictions on $h$ or $\tau$.
\begin{lemma}\label{lem:flowmaps}
Let $\muvec\nn$ and $\muGvec\nn$ be defined via \eqref{eq:mu:step2}, \eqref{eq:def:muInner}, and \eqref{eq:def:muG}.
Then the identity 
\begin{align*}
 \mobG\MGamma^{-1}\LGamma\muGvec\nn-\mobO\restr{\MOmega^{-1}}{\GammaSymb\times\GammaSymb}\restr{\LOmega}{\GammaSymb\times\OmegaSymb}\muvec\nn=0
\end{align*}
holds true.
\end{lemma}
\begin{proof}
Using \eqref{eq:def:muInner}, we compute
\begin{align}
\begin{split}
\mobG&\MGamma^{-1}\LGamma\muGvec\nn-\mobO\restr{\MOmega^{-1}}{\GammaSymb\times\GammaSymb}\restr{\LOmega}{\GammaSymb\times\OmegaSymb}\muvec\nn \\
=& -\mobG\MGamma^{-1}\LGamma\MGamma^{-1}\restr{\MOmega}{\GammaSymb\times\GammaSymb}\restr{\muvec\nn}{\GammaSymb}+ \mobG\MGamma^{-1}\LGamma\MGamma^{-1}\RHSG\trkla{\Phi\nn} \\
&-\mobO\restr{\MOmega^{-1}}{\GammaSymb\times\GammaSymb}\restr{\LOmega}{\GammaSymb\times\OmegaSymb} \muvec\nn\\
=&\mobG \MGamma^{-1}\LGamma\MGamma^{-1}\RHSG\trkla{\Phi\nn} \\
&- \restr{\MOmega^{-1}}{\GammaSymb\times\GammaSymb}\rkla{\mobO\restr{\LGamma}{\GammaSymb\times\GammaSymb}+\mobG\restr{\MOmega}{\GammaSymb\times\GammaSymb}\MGamma^{-1}\LGamma\MGamma^{-1}\restr{\MOmega}{\GammaSymb\times\GammaSymb}}\restr{\muvec\nn}{\GammaSymb}\\
&-\mobO\restr{\MOmega^{-1}}{\GammaSymb\times\GammaSymb}\restr{\LOmega}{\GammaSymb\times\InnerSymb}\restr{\muvec\nn}{\InnerSymb} =:I+II+III\,.
\end{split}
\end{align}
Recalling \eqref{eq:mu:step2} and \eqref{eq:def:muG}, we obtain
\begin{align}
\begin{split}
II=&\mobO\restr{\MOmega^{-1}}{\GammaSymb\times\GammaSymb}\restr{\LOmega}{\GammaSymb\times\InnerSymb}\restr{\MOmega^{-1}}{\InnerSymb\times\InnerSymb}\RHSO\trkla{\Phi\nn}-\mobG\MGamma^{-1}\LGamma\MGamma^{-1}\RHSG\trkla{\Phi\nn}\\
=&\mobO\restr{\MOmega^{-1}}{\GammaSymb\times\GammaSymb}\restr{\LOmega}{\GammaSymb\times\InnerSymb}\restr{\muvec\nn}{\InnerSymb}-\mobG\MGamma^{-1}\LGamma\MGamma^{-1}\RHSG\trkla{\Phi\nn}=-III-I\,,
\end{split}
\end{align}
which completes the proof.
\end{proof}
This result allows us to show that the phasefield parameter is conserved in $\Omega$ and on $\Gamma$.
Multiplying \eqref{eq:scheme} by $\ones^T\MOmega$ and by $\extend{\onesG^T\MGamma}{\OmegaSymb}$ proves the following corollary.
\begin{corollary}\label{cor:meanvalue}
Let $\Phi\nn$ be a discrete solution of \eqref{eq:scheme}. Then 
\begin{align*}
\ones^T\MOmega\Phi\nn = \ones^T\MOmega\Phi\no &&\onesG^T\MGamma\restr{\Phi\nn}{\GammaSymb}=\onesG^T\MGamma\restr{\Phi\no}{\GammaSymb}\,
\end{align*}
with $\ones:=\trkla{1,...,1}^T\in\R^{\dim\Uhs}$ and $\onesG:=\restr{\ones}{\GammaSymb}$.
\end{corollary}
Using the above auxiliary results, we are now able to state a first stability result which is a discrete version of the energy equality \eqref{eq:energy:CH}.
\begin{lemma}\label{lem:energy}
Let the assumptions \ref{item:disc:time}, \ref{item:disc:space}, \ref{item:disc:gamma}, \ref{item:potentials}, and \ref{item:potentialsbounds} hold true and let $\Phi\no\in\R^{\dim\Uhs}$ be given. Then a solution $\Phi\nn\in\R^{\dim\Uhs}$ to \eqref{eq:scheme}, if it exists, satifies
\begin{multline*}
\tfrac12{\Phi\nn}^T\LOmega\Phi\nn +\tfrac12\trkla{\Phi\nn-\Phi\no}^T\LOmega\trkla{\Phi\nn-\Phi\no} +\ones^T\MOmega F\trkla{\Phi\nn}+\tfrac12\kappa\restr{\Phi\nn}{\GammaSymb}^T\LGamma\restr{\Phi\nn}{\GammaSymb}\\
+\tfrac12\kappa\restr{\trkla{\Phi\nn-\Phi\no}}{\GammaSymb}^T\LGamma\restr{\trkla{\Phi\nn-\Phi\no}}{\GammaSymb}+\onesG^T\MGamma G\trkla{\restr{\Phi\nn}{\GammaSymb}}\\
+\antikappa\restr{\trkla{\Phi\nn-\Phi\no}}{\GammaSymb}^T\MGamma\restr{\trkla{\Phi\nn-\Phi\no}}{\GammaSymb}
+ \tau\mobO{\muvec\nn}^T\LOmega\muvec\nn +\tau\mobG{\muGvec\nn}^T\LGamma\muGvec\nn\\
\leq \tfrac12{\Phi\no}^T\LOmega\Phi\no +\ones^T\MOmega F\trkla{\Phi\no}+\tfrac12\kappa\restr{\Phi\no}{\GammaSymb}^T\LGamma\restr{\Phi\no}{\GammaSymb}+\onesG^T\MGamma G\trkla{\restr{\Phi\no}{\GammaSymb}}\,,
\end{multline*}
with $\ones:=\trkla{1,...,1}^T\in\R^{\dim\Uhs}$, $\onesG:=\restr{\ones}{\GammaSymb}$, and $\muvec\nn$ and $\muGvec\nn$ defined in \eqref{eq:def:muInner}, \eqref{eq:mu:step2}, and \eqref{eq:def:muG}.
\end{lemma}
\begin{proof}
We multiply \eqref{eq:scheme} by $\rkla{\MOmega\muvec\nn+\begin{psmallmatrix}\MGamma\muGvec\nn\\\bs{0}  \end{psmallmatrix}}$ and use Lemma \ref{lem:flowmaps} to obtain
\begin{align}
\begin{split}
0=&\rkla{\Phi\nn-\Phi\no}^T\MOmega\muvec\nn +\restr{\rkla{\Phi\nn-\Phi\no}}{\GammaSymb}^T\MGamma\muGvec\nn \\
&+ \tau\mobO\trkla{\muvec\nn}^T\LOmega\muvec\nn +\tau\mobG\trkla{\muGvec\nn}^T\LGamma\muGvec\nn\\
=:&\,I+II+III+IV\,.
\end{split}
\end{align}
As $III$ and $IV$ provide the dissipative parts of the desired estimate, we have show to that $I$ and $II$ yield the time difference of the energy.
Recalling \eqref{eq:def:muG}, we compute
\begin{align}
II=-\restr{\rkla{\Phi\nn-\Phi\no}}{\GammaSymb}^T\restr{\MOmega}{\GammaSymb\times\GammaSymb}\restr{\muvec\nn}{\GammaSymb} + \restr{\rkla{\Phi\nn-\Phi\no}}{\GammaSymb}^T\RHSG\trkla{\Phi\nn}\,.
\end{align}
Consequently, we obtain from \eqref{eq:mu:step2}
\begin{align}
\begin{split}
I+&II = \restr{\trkla{\Phi\nn-\Phi\no}}{\InnerSymb}^T\RHSO\trkla{\Phi\nn} + \restr{\trkla{\Phi\nn-\Phi\no}}{\GammaSymb}^T\RHSG\trkla{\Phi\nn}\\
=& \rkla{\Phi\nn-\Phi\no}^T\LOmega\Phi\nn+\trkla{\Phi\nn-\Phi\no}^T\MOmega\trkla{F_+^\prime\trkla{\Phi\nn}+F_-^\prime\trkla{\Phi\no}}\\
&+\kappa\restr{\trkla{\Phi\nn-\Phi\no}}{\GammaSymb}^T\LGamma\restr{\Phi\nn}{\GammaSymb} +\restr{\trkla{\Phi\nn-\Phi\no}}{\GammaSymb}^T\MGamma\trkla{G_+^\prime\trkla{\restr{\Phi\nn}{\GammaSymb}}+G^\prime_-\trkla{\restr{\Phi\no}{\GammaSymb}}}\,.
\end{split}
\end{align}
As $\MOmega$ and $\MGamma$ are diagonal matrices, we may combine $\trkla{F^\prime_+\trkla{\Phi\nn}+F^\prime_-\trkla{\Phi\no}}$ and  $\trkla{\Phi\nn-\Phi\no}$, and $\trkla{G^\prime_+\trkla{\restr{\Phi\nn}{\GammaSymb}}+G^\prime_-\trkla{\restr{\Phi\no}{\GammaSymb}}}$ and $\restr{\trkla{\Phi\nn-\Phi\no}}{\GammaSymb}$ componentwise.
In combination with $s_1\trkla{s_1-s_2} = \tfrac12 s_1^2+\tfrac12\trkla{s_1-s_2}^2-\tfrac12s_2^2$, this provides the result.
\end{proof}
Using the a priori estimate from Lemma \ref{lem:energy}, we are able to prove the existence of discrete solutions.
\begin{lemma}\label{lem:existence}
Let the assumptions \ref{item:disc:time}, \ref{item:disc:space}, \ref{item:disc:gamma}, and \ref{item:potentials} hold true and let $\Phi\no\in\R^{\dim\Uhs}$ be given.
Then, there exists at least one vector $\Phi\nn\in\R^{\dim\Uhs}$ solving \eqref{eq:scheme}.
\end{lemma}
\begin{proof}
We will prove the existence of discrete solutions by contradiction.
Let $\nnorm{.}$ denote the discrete $L^2$-norm which is derived from the inner product $\trkla{A,B}:=A^T\MOmega B$.
According to Corollary \ref{cor:meanvalue}, the mean-value of the phase-field is conserved in $\Omega$. 
This allows us to assume w.l.o.g. that $\ones^T\MOmega\Phi\nn=\ones^T\MOmega\Phi\no=0$.
Therefore, $\sqrt{{\Phi\nn}^T\LOmega\Phi\nn}$ is also a norm of $\Phi\nn$. 
Under the assumption that \eqref{eq:scheme} has no solution in 
\begin{align}
B_R:=\tgkla{ A\in\R^{\dim\Uhs}\,:\,\ones^T\MOmega A=0\text{~and~}\nnorm{A}\leq R}
\end{align}
for any $R>0$, the function $\mathcal{H}$ defined via
\begin{multline}
\Phi-\Phi\no +\tau\mobO\MOmega^{-1}\LOmega \begin{pmatrix}
\rkla{\mobO \restr{\LOmega}{\GammaSymb\times\GammaSymb} +\mobG\restr{\MOmega}{\GammaSymb\times\GammaSymb}\MGamma^{-1}\LGamma\MGamma^{-1} \restr{\MOmega}{\GammaSymb\times\GammaSymb}}^{-1} & \mb{0}\\ \mb{0}&\mathds{1}
\end{pmatrix}\\
\cdot\begin{pmatrix}
-\mobO\restr{\LOmega}{\GammaSymb\times\InnerSymb}\restr{\MOmega}{\InnerSymb\times\InnerSymb}^{-1}\RHSO\trkla{\Phi}+\mobG\restr{\MOmega}{\GammaSymb\times\GammaSymb}\MGamma^{-1}\LGamma\MGamma^{-1} \RHSG\trkla{\Phi}\\
\restr{\MOmega}{\InnerSymb\times\InnerSymb}^{-1}\RHSO\trkla{\Phi}
\end{pmatrix}=:\mathcal{H}\trkla{\Phi}
\end{multline}
has no root and is continuous on $B_R$.
This allows us to define a function $\mathcal{A}\,:B_R\rightarrow \partial B_R\subset B_R$ as
\begin{align}
\mathcal{A}\trkla{\Phi}:=-R\frac{\mathcal{H}\trkla{\Phi}}{\nnorm{\mathcal{H}\trkla{\Phi}}}\,.
\end{align} 
As $\mathcal{A}$ is continuous and maps a closed set onto itself, Brouwer's fixed point theorem provides the existence of at least one fixed point $\Phi^*$.
In the following, we will show 
\begin{align}\label{eq:contradiction}
0<\rkla{\Phi^*,\Psi}<0
\end{align}
for a suitable $\Psi\in B_R$ and $R$ large enough.
This contradiction shows that the initial assumption of \eqref{eq:scheme} not having solutions in $B_R$ is wrong.
To prove the contradiction \eqref{eq:contradiction}, we choose $\Psi=\tilde{\Psi}_1+\tilde{\Psi}_2-\ones^T\MOmega\trkla{\tilde{\Psi}_1+\tilde{\Psi}_2}\trkla{\ones^T\MOmega\ones}^{-1}\ones$ with
\begin{multline}
\tilde{\Psi}_1:=
\begin{pmatrix}
\rkla{\mobO \restr{\LOmega}{\GammaSymb\times\GammaSymb} +\mobG\restr{\MOmega}{\GammaSymb\times\GammaSymb}\MGamma^{-1}\LGamma\MGamma^{-1} \restr{\MOmega}{\GammaSymb\times\GammaSymb}}^{-1} & \mb{0}\\ \mb{0}&\mathds{1}
\end{pmatrix}\\
\cdot\begin{pmatrix}
-\mobO\restr{\LOmega}{\GammaSymb\times\InnerSymb}\restr{\MOmega}{\InnerSymb\times\InnerSymb}^{-1}\RHSO\trkla{\Phi^*}+\mobG\restr{\MOmega}{\GammaSymb\times\GammaSymb}\MGamma^{-1}\LGamma\MGamma^{-1} \RHSG\trkla{\Phi^*}\\
\restr{\MOmega}{\InnerSymb\times\InnerSymb}^{-1}\RHSO\trkla{\Phi^*}
\end{pmatrix}
\end{multline}
and 
\begin{align}
\tilde{\Psi}_2:=\begin{pmatrix}
 \restr{\MOmega^{-1}}{\GammaSymb\times\GammaSymb}\MGamma \rkla{-\MGamma^{-1} \restr{\MOmega}{\GammaSymb\times\GammaSymb} \restr{\tilde{\Psi}_1}{\Gamma} +\MGamma^{-1} \RHSG\trkla{\Phi^*}}\\
 \bs{0}
\end{pmatrix}\,,
\end{align}
i.e. the test vector is the sum of the chemical potentials deprived of their mean values.
The computations from the proof of Lemma \ref{lem:energy} provide
\begin{align}
\begin{split}
\rkla{\mathcal{H}\trkla{\Phi^*},\Psi}\geq& \tfrac12{\Phi^*}^T\LOmega{\Phi^*} -C\,,
\end{split}
\end{align}
where the constant $C$ depends on $\Phi\no$ and the lower bound from \ref{item:potentials}, but not on the fixed point $\Phi^*$ or $R$.
Since all norms on finite dimensional spaces are equivalent, there exists $c>0$ such that $\tfrac12{\Phi^*}^T\LOmega\Phi^*\geq c{\Phi^*}^T\MOmega\Phi^*$ and we obtain
\begin{align}
\rkla{\mathcal{H}\trkla{\Phi^*},\Psi}\geq c\nnorm{\Phi^*}^2-C=c R^2-C>0
\end{align}
for $R$ large enough.
This provides the second inequality in \eqref{eq:contradiction}.
In order to establish the first inequality we again use the computations from the proof of Lemma \ref{lem:energy} to show
\begin{align}
\begin{split}
\rkla{\Phi^*,\Psi} =&{\Phi^*}^T\LOmega\Phi^*+{\Phi^*}^T\MOmega\rkla{F_+^\prime\trkla{\Phi^*}+F_-^\prime\trkla{\bs{0}}} \\
&+{\Phi^*}^T\MOmega\rkla{F^\prime_-\trkla{\Phi\no}-F_-^\prime\trkla{\bs{0}}}+\kappa\restr{\Phi^*}{\GammaSymb}^T\LGamma\restr{\Phi^*}{\GammaSymb}\\
&+\restr{\Phi^*}{\GammaSymb}^T\MGamma\rkla{G_+^\prime\trkla{\restr{\Phi^*}{\GammaSymb}}+G^\prime_-\trkla{\bs{0}}}+\restr{\Phi^*}{\GammaSymb}^T\MGamma\rkla{G^\prime_-\trkla{\restr{\Phi\no}{\GammaSymb}}-G^\prime_-\trkla{\bs{0}}}\\
\geq& c\nnorm{\Phi^*}^2+\ones^T\MOmega\rkla{F\trkla{\Phi^*}-F\trkla{\bs{0}}}-\varepsilon\nnorm{\Phi^*}^2-C_\varepsilon\nnorm{F_-^\prime\trkla{\Phi\no}-F_-^\prime\trkla{\bs{0}}}^2\\
&+\kappa\restr{\Phi^*}{\GammaSymb}^T\LGamma\restr{\Phi^*}{\GammaSymb}+\onesG^T\MGamma\rkla{G\trkla{\restr{\Phi^*}{\GammaSymb}}-G\trkla{\bs{0}}} -\tilde{\varepsilon}\restr{\Phi^*}{\GammaSymb}^T\MGamma\restr{\Phi^*}{\GammaSymb}\\
&-C_{\tilde{\varepsilon}}\rkla{G_-^\prime\trkla{\restr{\Phi\no}{\GammaSymb}}-G_-^\prime\trkla{\bs{0}}}^T\MGamma\rkla{G_-^\prime\trkla{\restr{\Phi\no}{\GammaSymb}}-G_-^\prime\trkla{\bs{0}}}\,
\end{split}
\end{align}
with $0<\varepsilon,\tilde{\varepsilon}<\!\!<1$.
For every fixed $h$, there is a constant $C\h>0$ such that $\restr{\Phi^*}{\GammaSymb}^T\MGamma\restr{\Phi^*}{\GammaSymb}\leq C\h {\Phi^*}^T\MOmega\Phi^*$. Hence, we have
\begin{align*}
\rkla{\Phi^*,\Psi}\geq \trkla{c-\varepsilon -C\h\tilde{\varepsilon}}\nnorm{\Phi^*}^2 -C_{\varepsilon,\tilde{\varepsilon}}=\trkla{c-\varepsilon -C\h\tilde{\varepsilon}}R^2 -C_{\varepsilon,\tilde{\varepsilon}}
\end{align*}
with $C_{\varepsilon,\tilde{\varepsilon}}>0$ independent of $\Phi^*$ and $R$.
Choosing $\varepsilon$ and $\tilde{\varepsilon}$ small enough provides $\trkla{c-\varepsilon-C\h\tilde{\varepsilon}}>0$.
Hence, we obtain the first inequality in \eqref{eq:contradiction} for $R$ large enough, which completes the proof.
\end{proof}
\begin{remark}
The existence result in Lemma \ref{lem:energy} implies no constraints on the time increment $\tau$. 
Therefore, we have the existence of discrete solutions for arbitrary time increments.
\end{remark}

\section{Convergence of the discrete scheme}\label{sec:convergence}
In this section, we show that the discrete solutions established in the last section converge towards suitable weak solutions of \eqref{eq:model}.
This requires some assumptions on the initial data.
In particular, we will assume that
\begin{itemize}
\labitem{(I)}{item:initial} the initial data $\phi_0\in\Xkappa$ and its projection $\phi\h^0$ onto $\Uhs$ satisfies
\begin{align*}
\iOmega\abs{\nabla\phi\h^0}^2+\iOmega\Ih{F\trkla{\phi\h^0}} +\kappa\iGamma\abs{\nablaG\phi\h^0}^2 +\iGamma\IhG{G\trkla{\phi\h^0}}\leq C\,
\end{align*}
with some $C>0$ independent of $h$ and $\tau$.
\end{itemize}
Furthermore, the regularity results provided in this section require additional assumptions on $h$ and $\tau$.
In particular, we will need 
\begin{itemize}
\labitem{(C)}{item:htau} that $\tfrac{h^4}{\tau}\searrow 0$ for $\trkla{h,\tau}\searrow0$ when $\kappa>0$ and that $\tfrac{h^2}{\tau}\searrow 0$ for $\trkla{h,\tau}\searrow0$ when $\kappa=0$.
\end{itemize}
Assumption \ref{item:initial} allows us to state our first regularity result.
\begin{corollary}\label{cor:stability}
Let the assumptions \ref{item:disc:time}, \ref{item:disc:space}, \ref{item:disc:gamma}, \ref{item:potentials}, \ref{item:potentialsbounds}, and \ref{item:initial} hold true and let $h>0$ be small enough.
Then a solution $\trkla{\phi\h\nn,\muOh\nn,\muGh\nn}_{n=1,...,N}$ to \eqref{eq:feform} satisfies
\begin{multline*}
\max_{n=0,...,N}\norm{\phi\h\nn}_{H^1\trkla{\Omega}}^2 +\max_{n=0,...,N}\iOmega\Ih{F\trkla{\phi\h\nn}} +\kappa \max_{n=0,...,N}\norm{\phi\h\nn}_{H^1\trkla{\Gamma}}^2\\
+\max_{n=0,...,N}\iGamma\IhG{G\trkla{\phi\h\nn}}
+\sum_{n=1}^N\iOmega\abs{\nabla\phi\h\nn-\nabla\phi\h\no}^2 +\kappa\sum_{n=1}^N\iGamma\abs{\nablaG\phi\h\nn-\nablaG\phi\h\no}^2\\
+\antikappa\sum_{n=1}^N\iGamma{\abs{\phi\h\nn-\phi\h\no}^2}
+\tau\mobO\sum_{n=1}^N\norm{\muOh\nn}^2_{H^1\trkla{\Omega}} +\tau\mobG\sum_{n=1}^N\norm{\muGh\nn}^2_{H^1\trkla{\Gamma}}\leq C\,,
\end{multline*}
with a constant $C>0$ independent of $h$ and $\tau$.
\end{corollary}
\begin{proof}
After summing the result of Lemma \ref{lem:energy} over all time steps and recalling Corollary \ref{cor:meanvalue} and \ref{item:initial}, it remains to show that we have indeed control over the complete $H^1$ norm of $\muOh\nn$ and $\muGh\nn$.
To establish this result, we will follow the lines of \cite{GarckeKnopf20}. 
Testing \eqref{eq:feform:mu} by $\Ih{\eta}$ with $\eta\in C_0^\infty\trkla{\Omega;\tekla{0,1}}$, which is not identically zero, we obtain
\begin{align}
\iOmega\Ih{\muOh\nn\eta}=\iOmega\nabla\phi\h\nn\cdot\nabla\Ih{\eta} + \iOmega\Ih{\trkla{F_+^\prime\trkla{\phi\h\nn}+F_-^\prime\trkla{\phi\h\no}}\eta}\,.
\end{align}
From \ref{item:potentialsbounds} and \ref{item:initial}, we obtain
\begin{multline}\label{eq:bound:Fprime}
\abs{\iOmega\Ih{\trkla{F_+^\prime\trkla{\phi\h\nn}+F_-^\prime\trkla{\phi\h\no}}\eta}}\\
\leq C\norm{\phi\h\nn}_{L^3\trkla{\Omega}}^3 +C\norm{\phi\h\nn}_{L^1\trkla{\Omega}} +C\norm{\phi\h\no}_{L^3\trkla{\Omega}}^3+C\norm{\phi\h\no}_{L^1\trkla{\Omega}} +C\leq C
\end{multline}
Hence, there exists a constant $\tilde{C}\trkla{\eta}$ independent of $h$ and $\tau$ such that $\abs{\iOmega\Ih{\muOh\nn\eta}}\leq \tilde{C}\trkla{\eta}$.
We now define 
\begin{align}
\mathcal{M}_\eta:=\gkla{v\in H^1\trkla{\Omega}\,:\,\iOmega\Ih{v\eta}\leq \tilde{C}\trkla{\eta}}\,.
\end{align}
From standard error estimates for the interpolation operator $\Ihop$ (cf. \cite{BrennerScott}), we derive the existence of $c\trkla{\eta}>0$ such that $\iOmega\Ih{\eta}\geq c\trkla{\eta}$ for $h$ small enough.
Therefore, we may use the generalized Poincar\'e inequality (cf. \cite{Alt2016}), which we cite in the appendix as Lemma \ref{lem:poincare}, with $u_0\equiv0$ and $C_0:=\tilde{C}\trkla{\eta}/c\trkla{\eta}$ to obtain
\begin{align}
\norm{\muOh\nn}_{L^2\trkla{\Omega}}\leq C\trkla{1+\norm{\nabla\muOh\nn}_{L^2\trkla{\Omega}}}&&\text{for~all~}n\in\tgkla{1,...,N}\,.
\end{align}
To obtain the $L^2$-bound for $\muGh\nn$, we test \eqref{eq:feform:mu} by $\theta\h\equiv1$ and obtain
\begin{align*}
\abs{\iGamma\muGh\nn}\!\leq\! \abs{\iOmega\muOh\nn}\!+\!\abs{\iOmega\Ih{F_+^\prime\trkla{\phi\h\nn}+F_-^\prime\trkla{\phi\h\no}}}\!+\!\abs{\iGamma\IhG{G_+^\prime\trkla{\phi\h\nn}+G_-^\prime\trkla{\phi\h\no}}}\,.
\end{align*}
Considerations similar to \eqref{eq:bound:Fprime} show that the last term on the right-hand side is also bounded by a constant independent of $h$ and $\tau$.
Therefore, we may use Poincar\'e's inequality to complete the proof.
\end{proof}
In a second step, we derive uniform bounds for the time difference quotient $\dtau \phi\h\nn:=\tau^{-1}\trkla{\phi\h\nn-\phi\h\no}$ of the phase-field parameter on $\Omega$ and $\Gamma$.
\begin{lemma}\label{lem:timeder}
Let the assumptions \ref{item:disc:time}, \ref{item:disc:space}, \ref{item:disc:gamma}, \ref{item:potentials}, \ref{item:potentialsbounds},  \ref{item:initial}, and \ref{item:htau} hold true.
Furthermore, let $h>0$ be small enough such that Corollary \ref{cor:stability} holds true.
Then a solution $\trkla{\phi\h\nn}_{n=1,...,N}$ to \eqref{eq:feform} satisfies
\begin{align}\label{eq:timeder}
\tau\sum_{n=1}^N\norm{\dtau\phi\h\nn}_{\trkla{H^1\trkla{\Omega}}^\prime}^2\leq C\,,&&\text{and}&&
\tau\sum_{n=1}^N\norm{\dtau\phi\h\nn}_{\trkla{H^1\trkla{\Gamma}}^\prime}^2\leq C\,,
\end{align}
with $C>0$ independent of $h$ and $\tau$.
\end{lemma} 
\begin{proof}
We take $\theta\in H^1\trkla{\Omega}$ and test \eqref{eq:feform:Omega:phi} by $\theta\h:=\projUhsop\theta$, where $\projUhsop$ is the orthogonal $L^2$-projection onto $\Uhs$.
We decompose the first term in \eqref{eq:feform:Omega:phi} into
\begin{align}\label{eq:ihtime1}
\iOmega\Ih{\dtau\phi\h\nn\theta\h}=\iOmega\dtau\phi\h\nn\theta -\iOmega\trkla{\ids-\Ihop}\gkla{\dtau\phi\h\nn\theta\h}\,.
\end{align}
The first term will be used to obtain a norm on the dual space of $H^1\trkla{\Omega}$.
The second term can be controlled via Lemma \ref{lem:ihfe} and the $H^1$-stability of $\projUhsop$ (cf. \cite{BramblePasciakSteinbach2002}).
Using these considerations and Hölder's inequality, we obtain
\begin{align}\label{eq:ihtime2}
\abs{\iOmega\dtau\phi\h\nn\theta}\leq C\tfrac{h^2}{\tau}\norm{\nabla\phi\h\nn-\nabla\phi\h\no}_{L^2\trkla{\Omega}}\norm{\theta}_{H^1\trkla{\Omega}} + \norm{\nabla\muOh\nn}_{L^2\trkla{\Omega}}\norm{\theta}_{H^1\trkla{\Omega}}\,.
\end{align}
Dividing by $\norm{\theta}_{H^1\trkla{\Omega}}$, taking the second power on both sides, multiplying by $\tau$, and summing over all time steps provides 
\begin{align}
\tau\sum_{n=1}^N\norm{\dtau\phi\h\nn}_{\trkla{H^1\trkla{\Omega}}^\prime}^2\leq C\tfrac{h^4}{\tau}\sum_{n=1}^N\norm{\nabla\phi\h\nn-\nabla\phi\h\no}_{L^2\trkla{\Omega}}^2 +C\tau\sum_{n=1}^N\norm{\nabla\muOh\nn}_{L^2\trkla{\Omega}}^2\,.
\end{align}
Applying the already established regularity results and \ref{item:htau} completes the proof of the left inequality in \eqref{eq:timeder}.
For the case $\kappa>0$, the right inequality in \eqref{eq:timeder} can be established using similar computations.
In the case $\kappa=0$, we combine Lemma \ref{lem:ihfe} with an inverse estimate and obtain 
\begin{align}
\tau\sum_{n=1}^N\norm{\dtau\phi\h\nn}_{\trkla{H^1\trkla{\Gamma}}^\prime}^2\leq C\tfrac{h^2}{\tau}\sum_{n=1}^N\norm{\phi\h\nn-\phi\h\no}_{L^2\trkla{\Gamma}}^2 +C\tau\sum_{n=1}^N\norm{\nabla\muGh\nn}_{L^2\trkla{\Gamma}}^2\,.
\end{align}
Again, the already established regularity results and \ref{item:htau} complete the proof.
\end{proof}
In order to pass to the limit $\trkla{h,\tau}\searrow0$, we define time-interpolants of time-discrete functions $a\nn$, $\n=0,...,N$, and introduce some time-index-free notation as follows.\\
\begin{subequations}\label{eq:timeindexfree}
\begin{align}
a\tl\trkla{.,t}&:=\tfrac{t-t\no}{\tau}a\nn\trkla{.}+\tfrac{t\nn-t}{\tau} a\no\trkla{.} &&t\in\tekla{t\no,t\nn},n\geq1\,,\\
a\tp\trkla{.,t}&:=a\nn\trkla{.},\ \ a\tm\trkla{.,t}:=a\no\trkla{.}&&t\in(t\no,t\nn], n\geq1\,.
\end{align}
\end{subequations}
We want to point out that the time derivative of $a\tl$ coincides with the difference quotient, i.e.
\begin{align}
\para{t} a\tl= \para{t}\rkla{\tfrac{t-t\no}{\tau}a\nn+\tfrac{t\nn-t}{\tau} a\no} = \tfrac1\tau a\nn-\tfrac1\tau a\no =\dtau a\nn\,.
\end{align}
If a statement is valid for $a\tl$, $a\tp$, and $a\tm$, we use the abbreviation $a\tpm$.
With this notation, system \eqref{eq:feform} reads as follows.
\begin{subequations}\label{eq:timecont}
\begin{align}
\iOmegaT\Ih{\para{t}\phi\h\tl \theta\h} +\mobO\iOmegaT\nabla\muOh\tp\cdot\nabla\theta\h =& 0\,,\label{eq:timecont:Omega:phi}\\
\iGammaT\IhG{\para{t}\phi\h\tl \theta\h} + \mobG\iGammaT\nablaG\muGh\tp\cdot\nablaG\theta\h =& 0\,,\label{eq:timecont:Gamma:phi}
\end{align}
\begin{multline}
\iOmegaT\Ih{\muOh\tp\theta\h} + \iGammaT\IhG{\muGh\tp\theta\h} = \iOmegaT\nabla\phi\h\tp\cdot\nabla\theta\h \\
+\iOmegaT \Ih{\rkla{F_+^\prime\trkla{\phi\h\tp}+F_-^\prime\trkla{\phi\h\tm}}\theta\h} \\
+\kappa\iGammaT\nablaG\phi\h\tp\cdot\nablaG\theta\h +\iGammaT \IhG{\rkla{G_+^\prime\trkla{\phi\h\tp}+G_-^\prime\trkla{\phi\h\tm}}\theta\h} \label{eq:timecont:mu}
\end{multline}
\end{subequations}
for all $\theta\h\in L^2\trkla{0,T;\Uhs}$.
Similarly, we can rewrite the regularity results obtained in Corollary \ref{cor:stability} and Lemma \ref{lem:timeder} as
\begin{subequations}
\begin{multline}\label{eq:bounds}
\norm{\phi\h\tpm}_{L^\infty\trkla{0,T;H^1\trkla{\Omega}}}^2 +\kappa\norm{\phi\h\tpm}_{L^\infty\trkla{0,T;H^1\trkla{\Gamma}}}^2\\
 + \tau^{-1}\norm{\nabla\phi\h\tp-\nabla\phi\h\tm}_{L^2\trkla{0,T;L^2\trkla{\Omega}}}^2 +\kappa\tau^{-1}\norm{\nablaG\phi\h\tp-\nablaG\phi\h\tm}_{L^2\trkla{0,T;L^2\trkla{\Gamma}}}^2\\
 +\antikappa\tau^{-1}\norm{\phi\h\tp-\phi\h\tm}_{L^2\trkla{0,T;L^2\trkla{\Gamma}}}^2 
+\norm{\muOh\tp}_{L^2\trkla{0,T;H^1\trkla{\Omega}}}^2 +\norm{\muGh\tp}_{L^2\trkla{0,T;H^1\trkla{\Gamma}}}^2\leq C\,,
\end{multline}
as well as
\begin{align}\label{eq:boundsb}
\norm{\para{t}\phi\h\tl}_{L^2\trkla{0,T;\trkla{H^1\trkla{\Omega}}^\prime}}^2&\leq C&&\text{and}&&
\norm{\para{t}\phi\h\tl}_{L^2\trkla{0,T;\trkla{H^1\trkla{\Gamma}}^\prime}}^2\leq C\,.
\end{align}
\end{subequations}
These regularity results can be used to identify converging subsequences.
\begin{lemma}
Let the assumptions \ref{item:disc:time}, \ref{item:disc:space}, \ref{item:disc:gamma}, \ref{item:potentials}, \ref{item:potentialsbounds},  \ref{item:initial}, and \ref{item:htau} hold true.
Furthermore, let $\trkla{\phi\h\tpm,\,\muOh\tp,\,\muGh\tp}$ be a solution to \eqref{eq:timecont}.
Then there exists a subsequence (again denoted by $\trkla{\phi\h\tpm,\,\muOh\tp,\,\muGh\tp}$) and functions 
\begin{subequations}
\begin{align}
\phi&\in L^\infty\trkla{0,T;H^1\trkla{\Omega}}\cap H^1\trkla{0,T;\trkla{H^1\trkla{\Omega}}^\prime}\,,\\
\psi&\in \left\{\begin{matrix}
L^\infty\trkla{0,T;H^1\trkla{\Gamma}}\cap H^1\trkla{0,T;\trkla{H^1\trkla{\Gamma}}^\prime}&\text{~if~}\kappa>0\,,\\
L^\infty\trkla{0,T;H^{1/2}\trkla{\Gamma}}\cap H^1\trkla{0,T;\trkla{H^1\trkla{\Gamma}}^\prime}&\text{~if~}\kappa=0\,,
\end{matrix}\right.\\
\muO&\in L^2\trkla{0,T;H^1\trkla{\Omega}}\,,\\
\muG&\in L^2\trkla{0,T;H^1\trkla{\Gamma}}\,
\end{align}
\end{subequations}
such that $\trace{\phi}=\psi$ almost everywhere on $\Gamma_T$ and for $\trkla{h,\tau}\searrow0$
\begin{subequations}
\begin{align}
\phi\h\tpm&\weakstar\phi&&\text{in~} L^\infty\trkla{0,T;H^1\trkla{\Omega}}\,,\label{eq:conv:Omega:phi:weakstar}\\
\para{t}\phi\h\tl&\weak\para{t}\phi&&\text{in~} L^2\trkla{0,T;\trkla{H^1\trkla{\Omega}}^\prime}\label{eq:conv:Omega:phi:weak}\,,\\
\phi\h\tpm&\rightarrow\phi&&\text{in~} L^p\trkla{0,T,L^s\trkla{\Omega}}\,,\label{eq:conv:Omega:phi:strong}\\
\trace{\phi\h\tpm}&\weakstar\psi&&\text{in~} \left\{\begin{matrix}
L^\infty\trkla{0,T;H^1\trkla{\Gamma}}&\text{if~}\kappa>0\,,\\
L^\infty\trkla{0,T;H^{1/2}\trkla{\Gamma}} &\text{if~}\kappa=0\,,
\end{matrix}\right.\label{eq:conv:Gamma:phi:weakstar}\\
\para{t}\trace{\phi\h\tl}&\weak\para{t}\psi&&\text{in~}
L^2\trkla{0,T;\trkla{H^1\trkla{\Gamma}}^\prime}\,,\label{eq:conv:Gamma:phi:weak}\\
\trace{\phi\h\tpm}&\rightarrow\psi&&\text{in~}\left\{\begin{matrix}
L^p\trkla{0,T;L^q\trkla{\Gamma}}&\text{if~}\kappa>0\,,\\
L^p\trkla{0,T;L^{\tilde{s}}\trkla{\Gamma}}&\text{if~}\kappa=0\,,
\end{matrix}\right.\label{eq:conv:Gamma:phi:strong}\\
\muOh\tp&\weak \muO&&\text{in~}L^2\trkla{0,T;H^1\trkla{\Omega}}\,,\label{eq:conv:Omega:mu:weak}\\
\muGh\tp&\weak \muG&&\text{in~}L^2\trkla{0,T;H^1\trkla{\Gamma}}\label{eq:conv:Gamma:mu:weak}
\end{align}
for all $p<\infty$, $s\in[1,\tfrac{2d}{d-2})$, $q<\infty$ and $\tilde{s}\in[1,\tfrac{2(d-1)}{d-2})$.
\end{subequations}
\end{lemma}
\begin{proof}
The weak and weak$^*$ convergence expressed in \eqref{eq:conv:Omega:phi:weakstar}, \eqref{eq:conv:Omega:phi:weak}, \eqref{eq:conv:Omega:mu:weak}, and \eqref{eq:conv:Gamma:mu:weak} follows directly from the bounds in \eqref{eq:bounds} and \eqref{eq:boundsb}.
The strong convergence in \eqref{eq:conv:Omega:phi:strong} then follows from the bounds for $\phi\h\tpm$ in $L^\infty\trkla{0,T;H^1\trkla{\Omega}}$, the bounds on $\para{t}\phi\h\tl$ in $L^2\trkla{0,T;\trkla{H^1\trkla{\Omega}}^\prime}$, the Aubin--Lions theorem, and the fact that $\phi\h\tp$, $\phi\h\tm$, and $\phi\h\tl$ converge towards the same limit function due to the bound on $\tau^{-1}\norm{\nabla\phi\h\tp-\nabla\phi\h\tm}_{L^2\trkla{0,T;L^2\trkla{\Omega}}}^2$.\\
Similar arguments provide \eqref{eq:conv:Gamma:phi:weakstar}-\eqref{eq:conv:Gamma:phi:strong} in the case $\kappa>0$.
In the case $\kappa=0$, we use the uniform bound on $\norm{\phi\h\nn}_{H^1\trkla{\Omega}}$ to deduce a uniform bound for $\norm{\trace{\phi\h\nn}}_{H^{\trkla{1/2}}\trkla{\Gamma}}$.
As $H^{1/2}\trkla{\Gamma}$ is compactly embedded in $L^{\tilde{s}}\trkla{\Gamma}$ for $\tilde{s}\in[1,\tfrac{2\trkla{d-1}}{d-2})$ (cf. \cite{HitchhikersGuide}), we verify \eqref{eq:conv:Gamma:phi:weakstar}-\eqref{eq:conv:Gamma:phi:strong} for $\kappa=0$.
It remains to show that $\psi$ can be identified with $\trace{\phi}$.
We choose $\bs{\theta}\in L^2\trkla{0,T;\trkla{C^\infty{\trkla{\Omega}}}^d}$ and compute
\begin{multline}
\iOmegaT\phi\div\bs{\theta}\leftarrow\iOmegaT\phi\h\tpm\div\bs{\theta}=-\iOmegaT\nabla\phi\h\tpm\cdot\bs{\theta} +\iGammaT\trace{\phi\h\tpm}\bs{\theta}\cdot\bs{n}\\
\rightarrow -\iOmegaT\nabla\phi\cdot\bs{\theta} +\iGammaT\psi\bs{\theta}\cdot\bs{n}=\iOmegaT\phi\div\bs{\theta}-\iGammaT\trace{\phi}\bs{\theta}\cdot\bs{n}+\iGammaT\psi\bs{\theta}\cdot\bs{n}\,.
\end{multline}
\end{proof}
\begin{theorem}\label{th:convergence}
Let $d\in\tgkla{2,3}$ and let the assumptions \ref{item:disc:time}, \ref{item:disc:space}, \ref{item:disc:gamma}, \ref{item:potentials}, \ref{item:potentialsbounds},  \ref{item:initial}, and \ref{item:htau} hold true.
Then a tuple $\trkla{\phi,\mu,\muG}$ satisfying
\begin{subequations}
\begin{align}
\phi&\in L^\infty\trkla{0,T;H^1\trkla{\Omega}}\cap H^1\trkla{0,T;\trkla{H^1\trkla{\Omega}}^\prime}\,,\\
\trace{\phi}&\in \left\{\begin{matrix}
L^\infty\trkla{0,T;H^1\trkla{\Gamma}}\cap H^1\trkla{0,T;\trkla{H^1\trkla{\Gamma}}^\prime}&\text{~if~}\kappa>0\,,\\
L^\infty\trkla{0,T;H^{1/2}\trkla{\Gamma}}\cap H^1\trkla{0,T;\trkla{H^1\trkla{\Gamma}}^\prime}&\text{~if~}\kappa=0\,,
\end{matrix}\right.\\
\muO&\in L^2\trkla{0,T;H^1\trkla{\Omega}}\,,\\
\muG&\in L^2\trkla{0,T;H^1\trkla{\Gamma}}\,
\end{align}
\end{subequations}
can be obtained from discrete solutions to \eqref{eq:feform} by passing to the limit $\trkla{h,\tau}\searrow0$.
This tuple solves \eqref{eq:model} in the following weak sense:
\begin{subequations}\label{eq:result}
\begin{align}
\int_0^T\skla{\para{t}\phi, \theta} +\mobO\iOmegaT\nabla\muO\cdot\nabla\theta =& 0&&\forall \theta\in L^2\trkla{0,T;H^1\trkla{\Omega}}\,,\label{eq:result:Omega:phi}\\
\int_0^T\skla{\para{t}\trace{\phi}, \theta}_{\Gamma} + \mobG\iGammaT\nablaG\muG\cdot\nablaG\theta =& 0&&\forall\theta\in L^2\trkla{0,T;H^1\trkla{\Gamma}}\,,\label{eq:result:Gamma:phi}
\end{align}
\begin{multline}
\iOmegaT\muO\theta + \iGammaT\muG\trace{\theta} = \iOmegaT\nabla\phi\cdot\nabla\theta +\iOmegaT F^\prime\trkla{\phi}\theta \\
+\kappa\iGammaT\nablaG\trace{\phi}\cdot\nablaG\trace{\theta} +\iGammaT G^\prime\trkla{\trace{\phi}}\trace{\theta} \qquad\qquad\forall \theta\in L^2\trkla{0,T;\Xkappa}\,.\label{eq:result:mu}
\end{multline}
\end{subequations}
\end{theorem}
\begin{proof}
We start by passing to the limit in \eqref{eq:timecont:Omega:phi}.
Choosing $\theta\h:=\Ih{\theta}$ for $\theta\in L^2\trkla{0,T;C^\infty\trkla{\overline{\Omega}}}$, we have $\theta\h\rightarrow \theta$ in $L^2\trkla{0,T;H^1\trkla{\Omega}}$ (cf. \cite{BrennerScott}).
We decompose the first term as
\begin{align}
\iOmegaT\Ih{\para{t}\phi\h\tl\theta\h}=\iOmegaT\para{t}\phi\h\tl\theta\h-\iOmegaT\trkla{\ids-\Ihop}\tgkla{\para{t}\phi\h\tl\theta}\,.
\end{align}
This allows us to combine the results from \eqref{eq:ihtime1} and \eqref{eq:ihtime2} with \eqref{eq:conv:Omega:phi:weak} and \eqref{eq:conv:Omega:mu:weak} to derive \eqref{eq:result:Omega:phi} for $\theta\in L^2\trkla{0,T;C^\infty\trkla{\overline{\Omega}}}$.
Noting that $L^2\trkla{0,T;C^\infty\trkla{\overline{\Omega}}}$ is dense in $L^2\trkla{0,T;H^1\trkla{\Omega}}$ yields the result.
Similar arguments allow us to pass to the limit in \eqref{eq:timecont:Gamma:phi} to obtain \eqref{eq:result:Gamma:phi}.\\
In order to pass to the limit in \eqref{eq:timecont:mu}, we choose $\theta\h\!:=\!\Ih{\theta}$ with $\theta\in L^2\trkla{0,T;C^\infty\trkla{\overline{\Omega}}}$ and assume that $\trace{\phi\h\tpm}\in L^\infty\trkla{0,T;H^{1/2}\trkla{\Gamma}}$, which is the case for $\kappa>0$ and $\kappa=0$. 
While the convergence of the left-hand side of \eqref{eq:timecont:mu} and the gradient terms is straightforward, the convergence of the terms including the derivative of the potential functions $F$ and $G$ require more finesse.
We will showcase the convergence of $\iGammaT\IhG{G_+^\prime\trkla{\phi\h\tp}\theta\h}$. 
Then, the convergence of the remaining parts can be obtained in an analogous manner.
According to \ref{item:potentialsbounds}, $G_+^\prime$ can be written as the sum of a polynomial of degree three and a globally Lipschitz-continuous component ${G_+^L}^\prime$.
We start with the decomposition
\begin{multline}
\label{eq:tmp:convergence}
\iGammaT\IhG{\trkla{\phi\h\tp}^3\theta\h} = \iGammaT\trkla{\phi\h\tp}^3\theta\h -\iGammaT\trkla{\ids-\IhGop}\gkla{\trkla{\phi\h\tp}^2}\phi\h\tp\theta\h \\
-\iGammaT\trkla{\ids-\IhGop}\gkla{\IhG{\trkla{\phi\h\tp}^2}\phi\h\tp}\theta\h -\iGammaT\trkla{\ids-\IhGop}\gkla{\IhG{\trkla{\phi\h\tp}^3}\theta\h}\,.
\end{multline}
The convergence of the first term on the right-hand side follows directly from \eqref{eq:conv:Gamma:phi:strong} and the strong convergence of $\theta\h\rightarrow\theta$.
Therefore, it remains to show that the remaining terms vanish when passing to the limit.
Recalling that $H^{1/2}\trkla{\Gamma}$ is continuously embedded in $L^4\trkla{\Gamma}$ (cf. \cite{HitchhikersGuide}), the estimates in Lemma \ref{lem:ihfe} and the standard inverse estimates (cf. \cite{BrennerScott}) provide
\begin{multline}
\abs{\iGamma\trkla{\ids-\IhGop}\gkla{\IhG{\trkla{\phi\h\tp}^3}\theta\h}}\leq C h\norm{\IhG{\trkla{\phi\h\tp}^3}}_{L^2\trkla{\Gamma}}\norm{\nablaG\theta\h}_{H^1\trkla{\Gamma}}\\
\leq Ch\norm{\phi\h\tp}_{L^6\trkla{\Gamma}}^3\norm{\nablaG\theta\h}_{H^1\trkla{\Gamma}}\leq Ch^{1/2}\norm{\phi\h\tp}_{L^4\trkla{\Gamma}}^3\norm{\nablaG\theta\h}_{H^1\trkla{\Gamma}}\,.
\end{multline}
Therefore, the last term in \eqref{eq:tmp:convergence} vanishes.
Furthermore, we derive the estimates
\begin{multline}
\abs{\iGamma\trkla{\ids-\IhGop}\gkla{\IhG{\trkla{\phi\h\tp}^2}\phi\h\tp}\theta\h}\\
\leq \norm{\trkla{\ids-\IhGop}\gkla{\IhG{\trkla{\phi\h\tp}^2}\phi\h\tp}}_{L^{5/4}\trkla{\Gamma}}\norm{\theta\h}_{H^1\trkla{\Gamma}}\\
\leq Ch^2\norm{\nabla\phi\h\tp}_{L^{10/3}\trkla{\Gamma}}\norm{\nabla\IhG{\trkla{\phi\h\tp}^2}}_{L^2\trkla{\Gamma}} \norm{\theta\h}_{H^1\trkla{\Gamma}}\\
\leq Ch^{2/5} \norm{\phi\h\tp}_{L^4\trkla{\Gamma}}^3 \norm{\theta\h}_{H^1\trkla{\Gamma}}\,
\end{multline}
and 
\begin{multline}
\abs{\iGamma\trkla{\ids-\IhGop}\gkla{\rkla{\phi\h\tp}^2}\phi\h\tp\theta\h}\\
\leq\norm{\trkla{\ids-\IhGop}\gkla{\rkla{\phi\h\tp}^2}}_{L^{3/2}\trkla{\Gamma}}\norm{\phi\h\tp}_{L^4\trkla{\Gamma}}\norm{\theta\h}_{H^1\trkla{\Gamma}}\\
\leq Ch^2\norm{\nablaG\phi\h\tp}_{L^3\trkla{\Gamma}}^2\norm{\phi\h\tp}_{L^4\trkla{\Gamma}}\norm{\theta\h}_{H^1\trkla{\Gamma}} \leq Ch^{1/6}\norm{\phi\h\tp}_{L^4\trkla{\Gamma}}^3\norm{\theta\h}_{H^1\trkla{\Gamma}}\,.
\end{multline}
As $\phi\h\tp\in L^p\trkla{0,T;L^4\trkla{\Gamma}}$, we obtain the convergence of the polynomial part of $G^\prime_+$.
To deal with the Lipschitz-continuous part ${G_+^L}^\prime$, we start with the decomposition
\begin{multline}
\iGammaT\IhG{{G_+^L}^\prime\trkla{\phi\h\tp}\theta\h}=\iGammaT {G_+^L}^\prime\trkla{\phi\h\tp}\theta\h \\
-\iGammaT\trkla{\ids-\IhGop}\gkla{{G_+^L}^\prime\trkla{\phi\h\tp}}\theta\h -\iGammaT\trkla{\ids-\IhGop}\gkla{\IhG{{G_+^L}^\prime\trkla{\phi\h\tp}}\theta\h}:= I+II+III\,.
\end{multline}
Combining Lemma \ref{lem:ihfe} with a standard inverse estimate, we compute
\begin{align}
\begin{split}
\abs{III}&\leq \int_0^TCh^2\norm{\nablaG\IhG{{G_+^L}^\prime\trkla{\phi\h\tp}}}_{L^2\trkla{\Gamma}}\norm{\nablaG\theta\h}_{L^2\trkla{\Gamma}}\\
&\leq \int_0^T Ch^{3/2}\norm{\IhG{{G_+^L}^\prime\trkla{\phi\h\tp}}}_{L^4\trkla{\Gamma}}\norm{\nablaG\theta\h}_{L^2\trkla{\Gamma}}\,.
\end{split}
\end{align}
Using the Lipschitz-continuity of ${G_+^L}^\prime$, we deduce 
\begin{multline}
\norm{{G_+^L}^\prime\trkla{\phi\h\tp}}_{L^\infty\trkla{0,T;L^4\trkla{\Gamma}}}+\norm{\IhG{{G_+^L}^\prime\trkla{\phi\h\tp}}}_{L^\infty\trkla{0,T;L^4\trkla{\Gamma}}}\\\leq C\norm{\phi\h\tp}_{L^\infty\trkla{0,T;L^4\trkla{\Gamma}}}+C\,,
\end{multline}
with a constant $C$ depending on the Lipschitz-constant of ${G_+^L}^\prime$.
Furthermore, the Lipschitz-continuity provides on every $K^\Gamma\in\Th^\Gamma$
\begin{multline}
\int_{K^\Gamma}\abs{\IhG{{G_+^L}^\prime\trkla{\phi\h\tp}}-{G_+^L}^\prime\trkla{\phi\h\tp}}^2\leq C\int_{K^\Gamma}\abs{\max_{K^\Gamma}\tgkla{\phi\h\tp}-\min_{K^\Gamma}\tgkla{\phi\h\tp}}^2\\
\leq Ch^2\int_{K^\Gamma}\abs{\nablaG\phi\h\tp}^2\,.
\end{multline}
Consequently, an inverse estimate yields
\begin{align}
\norm{\trkla{\ids-\IhGop}\gkla{{G_+^L}^\prime\trkla{\phi\h\tp}}}_{L^2\trkla{\Gamma}}\leq Ch\norm{\nablaG\phi\h\tp}_{L^2\trkla{\Gamma}}\leq Ch^{1/2}\norm{\phi\h\tp}_{L^4\trkla{\Gamma}}\,,
\end{align}
which proves that $II$ will also vanish when passing to the limit.
From the strong convergence \eqref{eq:conv:Gamma:phi:strong}, we deduce ${G_+^L}^\prime\trkla{\phi\h\tp}\rightarrow {G_+^L}^\prime\trkla{\trace{\phi}}$ almost everywhere.
Recalling ${G_+^L}^\prime\trkla{\phi\h\tp}\in L^\infty\trkla{0,T;L^4\trkla{\Gamma}}$, we may use Vitali's convergence theorem (see e.g. \cite{Alt2016}) to show ${G_+^L}^\prime\trkla{\phi\h\tp}\rightarrow {G_+^L}^\prime\trkla{\trace{\phi}}$ in $L^\infty\trkla{0,T;L^{\tilde{s}}\trkla{\Gamma}}$ for $\tilde{s}<4$.
The convergence of derivatives of the concave parts of $G$ follows from the same arguments.
The uniform bounds of $\phi\h\tpm$ in $L^\infty\trkla{0,T;H^1\trkla{\Omega}}$ provide enough regularity, to adapt the previously presented arguments to three spatial dimensions, which proves the convergence of the remaining terms.
As $C^\infty\trkla{\overline{\Omega}}$ is dense in $\Xkappa$, this concludes the proof.
\end{proof}
\begin{remark}\label{rem:ac}
The results presented in the preceding sections carry over to the case of Allen--Cahn-type dynamic boundary conditions (cf. \eqref{eq:bc:AC}), where we use 
\begin{align}\label{eq:feform:Gamma:phi:AC}
\iGamma\IhG{\dtau\phi\h\nn\theta\h}=-\mobG\iGamma\IhG{\muGh\nn\theta\h}&&\text{for~all~}\theta\h\in\Uhs\,.
\end{align}
instead of \eqref{eq:feform:Gamma:phi}.
The resulting scheme reads
\begin{multline}\label{eq:scheme:AC}
\Phi\nn +\tau\mobO\MOmega^{-1}\LOmega \begin{pmatrix}
\rkla{\mobO \restr{\LOmega}{\GammaSymb\times\GammaSymb} +\mobG\restr{\MOmega}{\GammaSymb\times\GammaSymb}\MGamma^{-1} \restr{\MOmega}{\GammaSymb\times\GammaSymb}}^{-1} & \mb{0}\\ \mb{0}&\mathds{1}
\end{pmatrix}\\
\cdot\begin{pmatrix}
-\mobO\restr{\LOmega}{\GammaSymb\times\InnerSymb}\restr{\MOmega}{\InnerSymb\times\InnerSymb}^{-1}\RHSO\trkla{\Phi\nn}+\mobG\restr{\MOmega}{\GammaSymb\times\GammaSymb}\MGamma^{-1}\LGamma\MGamma^{-1} \RHSG\trkla{\Phi\nn}\\
\restr{\MOmega}{\InnerSymb\times\InnerSymb}^{-1}\RHSO\trkla{\Phi\nn}
\end{pmatrix}=\Phi\no\,
\end{multline}
and is well defined, as $\rkla{\mobO \restr{\LOmega}{\GammaSymb\times\GammaSymb} +\mobG\restr{\MOmega}{\GammaSymb\times\GammaSymb}\MGamma^{-1} \restr{\MOmega}{\GammaSymb\times\GammaSymb}}$ is obviously a symmetric, positive definite matrix.\\
Although, $\iGamma\phi\h\nn$ is not conserved when using Allen--Cahn-type boundary conditions, testing \eqref{eq:feform:Gamma:phi:AC} by $1$ shows that $\abs{\iGamma\phi\h\nn}$ is bounded.
Consequently, the energy estimate still provides control over $\norm{\phi\h\nn}_{H^1\trkla{\Gamma}}$.\\
Testing \eqref{eq:feform:Gamma:phi:AC} by $\dtau\phi\h\nn$ shows $\tau\sum_{n=1}^N\norm{\dtau\phi\h\nn}_{L^2\trkla{\Gamma}}^2\leq C$, i.e. we obtain a slightly better regularity result for the discrete time derivative than we obtained for Cahn--Hilliard-type boundary conditions.
Using the time-index-free notation introduced in \eqref{eq:timeindexfree}, the bounds read
\begin{multline}\label{eq:bounds:AC}
\norm{\phi\h\tpm}_{L^\infty\trkla{0,T;H^1\trkla{\Omega}}}^2 +\kappa\norm{\phi\h\tpm}_{L^\infty\trkla{0,T;H^1\trkla{\Gamma}}}^2\\
 + \tau^{-1}\norm{\nabla\phi\h\tp-\nabla\phi\h\tm}_{L^2\trkla{0,T;L^2\trkla{\Omega}}}^2 +\kappa\tau^{-1}\norm{\nablaG\phi\h\tp-\nablaG\phi\h\tm}_{L^2\trkla{0,T;L^2\trkla{\Gamma}}}^2\\
 +\antikappa\tau^{-1}\norm{\phi\h\tp-\phi\h\tm}_{L^2\trkla{0,T;L^2\trkla{\Gamma}}}^2 
+\norm{\muOh\tp}_{L^2\trkla{0,T;H^1\trkla{\Omega}}}^2 +\norm{\muGh\tp}_{L^2\trkla{0,T;L^2\trkla{\Gamma}}}^2\\
+\norm{\para{t}\phi\h\tl}_{L^2\trkla{0,T;\trkla{H^1\trkla{\Omega}}^\prime}}^2+\norm{\para{t}\phi\h\tl}_{L^2\trkla{0,T;L^2\trkla{\Gamma}}}^2\leq C\,
\end{multline}
with $C>0$ independent of $h$ and $\tau$.
Based on these uniform bounds, we are able to identify converging subsequences and pass to the limit.\\

Variants of these boundary conditions are used to describe dynamic contact angles.
To recover the boundary condition suggested in \cite{Qian2006}, we choose $\kappa=0$, $\deltaG=1$, and
$G\trkla{\phi}= \tfrac\gamma2 \sin\trkla{\tfrac\pi2 \min\tgkla{\max\tgkla{\phi,-1},1}}+\tfrac{\tabs{\gamma}}2$, where the parameter $\gamma$ prescribes the static contact angle via Young's formula.
As $G$ satisfies \ref{item:potentials} and \ref{item:potentialsbounds} the previous results are also valid for the boundary condition \eqref{eq:dynamic_angle}.
\end{remark}

\section{Numerical simulations}\label{sec:simulation}
In this section, we present simulations to underline the practicality of the proposed scheme \eqref{eq:scheme}, which we implemented in the \texttt{C++} framework EconDrop (cf. \cite{Grun2013c,Campillo2012,GrunGuillenMetzger2016, Metzger2018,Metzger_2018b}).
This framework allows for adaptivity in space and time using the ideas presented in \cite{Grun2013c}, i.e. we are able to use meshes with a high resolution in the evolving interfacial area and a lower resolution in the bulk phases where $\phi\approx\pm1$. 
Similarly, the time increments can be varied such that they are small, when the solution changes rapidly and larger when the solution is almost stationary.
In the presented simulations, we used Newton's method to linearize \eqref{eq:scheme}, a biconjugate gradient stabilized method to solve the arising linear system, and a preconditioned conjugate gradient method to tackle the smaller auxiliary problem.
\subsection{Scenario 1}\label{subsec:separation}
As a first test case, we consider a phase-separation scenario in $\Omega=\trkla{0,1}^2$ starting from  the initial condition
\begin{align}
\phi\h^0\,:\,\R^2\supset\Omega\rightarrow \tekla{-1,+1}&&\trkla{x_1,x_2}\mapsto \Ih{0.1\sin\trkla{2\pi x_1}\sin\trkla{2\pi x_2}}\,.
\end{align}
To keep $\phi$ close to the physical meaningful interval $\tekla{-1,+1}$, we use the penalized double-well potential $W_{\operatorname{pen}}$ (cf. \eqref{eq:penpot}) with penalty parameter $C_{\text{pen}}=250$ for $F$ and $G$.
In this simulation, we use adaptive mesh refinement based on the criteria proposed in \cite{Grun2013c} using triangles with diameters between $\sqrt{2}\cdot2^{-6}$ and $2^{-8}$.
The dimensions of the corresponding finite element spaces are depicted in Fig.~\ref{fig:dim}.
Examples of the used triangulations can be found in Fig. \ref{fig:wire}.
The size of the time increment is also chosen adaptively based on the ideas presented in \cite{Grun2013c}, which leads to time increments between $6.3\cdot10^{-7}$ and $6.4\cdot10^{-4}$.
The remaining parameters are listed in Tab.~\ref{tab:params2}.\\
The evolution of the phase-separation process is depicted in Figure \ref{fig:scenario2}.
Thereby, the light color represents the phase $\phi=+1$ and the dark color represents $\phi=-1$ in all pictures except the first one.
As the initial data only contains values in $\tekla{-0.1,+0.1}$, we rescaled Fig. \ref{fig:scenario2:initial} such that the light color corresponds to $\phi=+0.1$, while the dark color corresponds to $\phi=-0.1$.\\
The phase-separation starts with the typical wave pattern (see Figs.~\ref{fig:scenario2:b}-\ref{fig:scenario2:d}), which then transforms into two entangled spirals (cf. Figs.~\ref{fig:scenario2:e}-\ref{fig:scenario2:f}).
In the course of the simulation, the spirals retreat (cf. Figs.~\ref{fig:scenario2:f}-\ref{fig:scenario2:l}) reducing the fluid-fluid contact area.
The corresponding decrease in energy of depicted in Fig.~\ref{fig:energy2}.\\
Fig.~\ref{fig:mass2} shows evolution of $\iOmega\phi\h\tl$ and $\iGamma\phi\h\tl$.
In theory, these quantities should be conserved.
In our simulation, the deviation of these quantities from their initial values is of order $10^{-4}$ and originates from the adaptivity of the used triangulation.\\
To compare \eqref{eq:scheme} with the straightforward approach of solving \eqref{eq:feform} directly in terms of practicality, we investigate the dependence of the condition numbers on the size of the time increment.
For this purpose, we computed the matrices needed in the first Newton step of both schemes for artificial time increments $\tau\in\{1\cdot10^{-4},\, 4\cdot10^{-5},\, 2\cdot10^{-5},\, 1\cdot10^{-5},\, 1\cdot10^{-6},\, 1\cdot10^{-7}\}$ and estimated their condition number using the \textsc{Matlab} function \texttt{condest}.
Averages of the condition numbers based on 51 equidistant points in time are plotted in Fig.~\ref{fig:cond2}.
As expected, the condition numbers in the straightforward approach grow for vanishing $\tau$.
Applying a Jacobi preconditioner to the system reduces the condition number drastically.
However, for vanishing $\tau$ the condition number of the preconditioned system still increases rapidly (cf. blue triangles in Fig.~\ref{fig:cond2}).
On the contrary, the condition number for \eqref{eq:scheme}, drops for decreasing $\tau$, as the matrix in \eqref{eq:scheme} converges towards the identity.\\
As discussed in Rem.~\ref{rem:efficiency}, solving \eqref{eq:scheme} requires us to solve a smaller auxiliary problem repeatedly.
In the discussed simulation, the average number of needed cg-iterations remained below 20 (cf.~Fig.~\ref{fig:averagecg}).\\
When developing our scheme, we boiled \eqref{eq:feform} down to \eqref{eq:scheme} using the fact that \eqref{eq:feform:Omega:phi} and \eqref{eq:feform:Gamma:phi} have to provide identical values for the trace of $\phi\h\nn$.
In order to validate our scheme, we use \eqref{eq:scheme} to compute the new phase-field values, recover $\muGh\nn$ via \eqref{eq:def:muG}, and check whether \eqref{eq:feform:Gamma:phi} still holds true.
In this simulation, the $L^2\trkla{\Gamma}$-norm of the deviation averages out to $3.3\cdot 10^{-9}$.
\begin{table}
\center
\begin{tabular}{c|c|c|c|c|c|c}
$\mobO$&$\delta$& $\sigma$& $\mobG$&$\deltaG$& $\kappa$ &$C_\text{pen}$\\
0.01&0.01&2&0.01&0.02&1&250
\end{tabular}
\caption{Parameters used in Section \ref{subsec:separation}.}
\label{tab:params2}
\end{table}
\begin{figure}
\newcommand{\scale}{.23\textwidth}
\subfloat[][$t=0$ (rescaled)]{
\includegraphics[width=\scale]{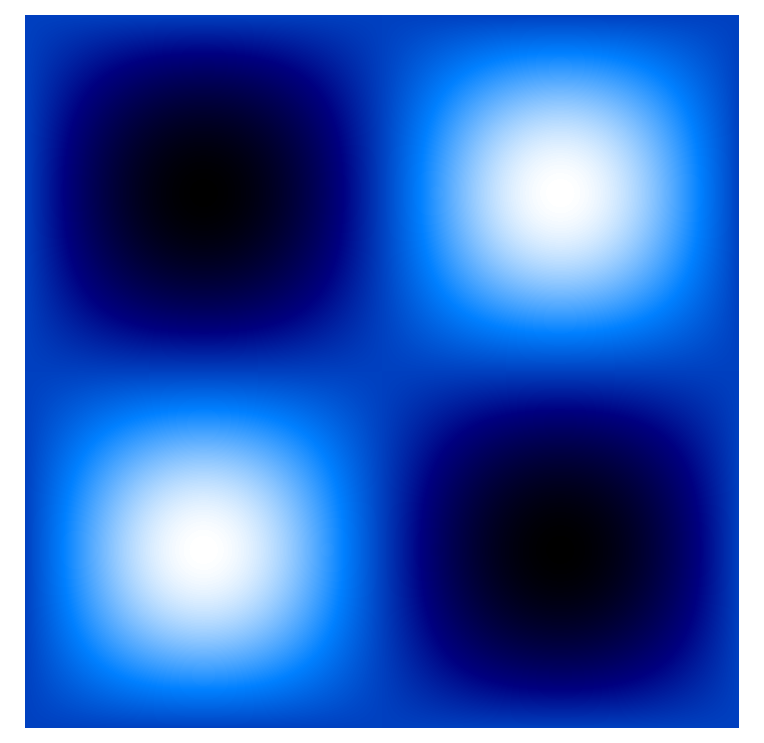}
\label{fig:scenario2:initial}
}\hfill
\subfloat[][$t=0.0022$]{
\includegraphics[width=\scale]{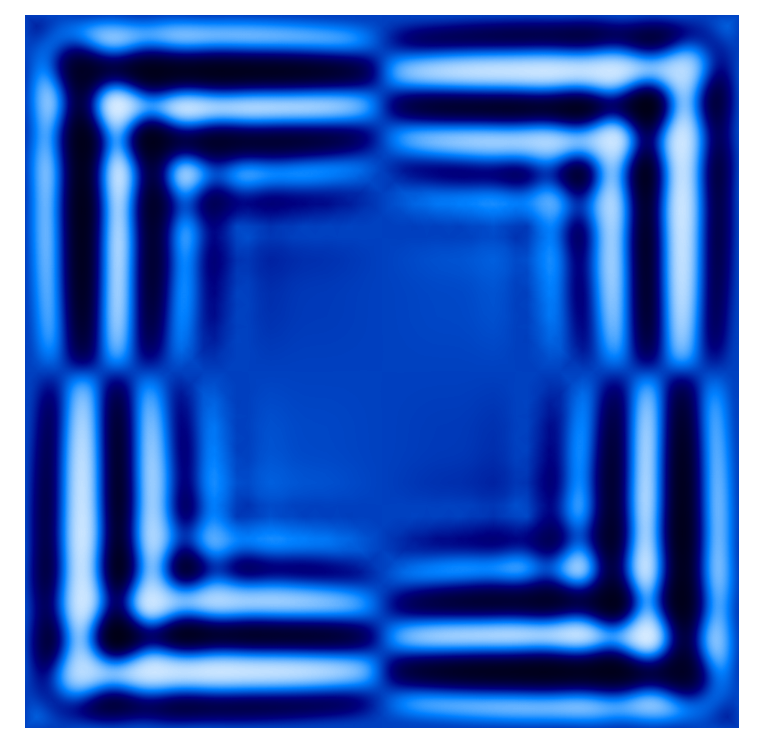}
\label{fig:scenario2:b}
}\hfill
\subfloat[][$t=0.0036$]{
\includegraphics[width=\scale]{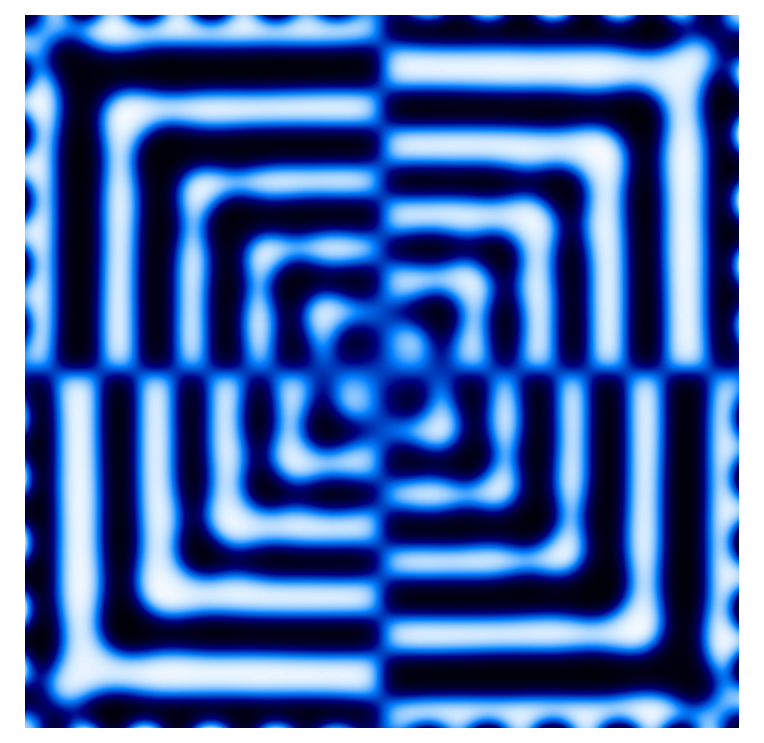}
}\hfill
\subfloat[][$t=0.0044$]{
\includegraphics[width=\scale]{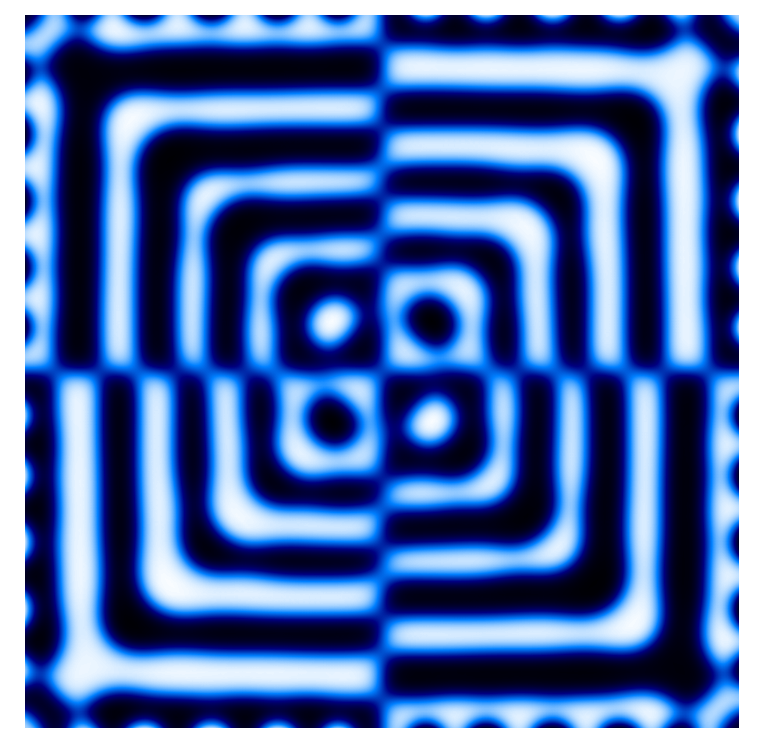}
\label{fig:scenario2:d}
}\\
\subfloat[][$t=0.0058$]{
\includegraphics[width=\scale]{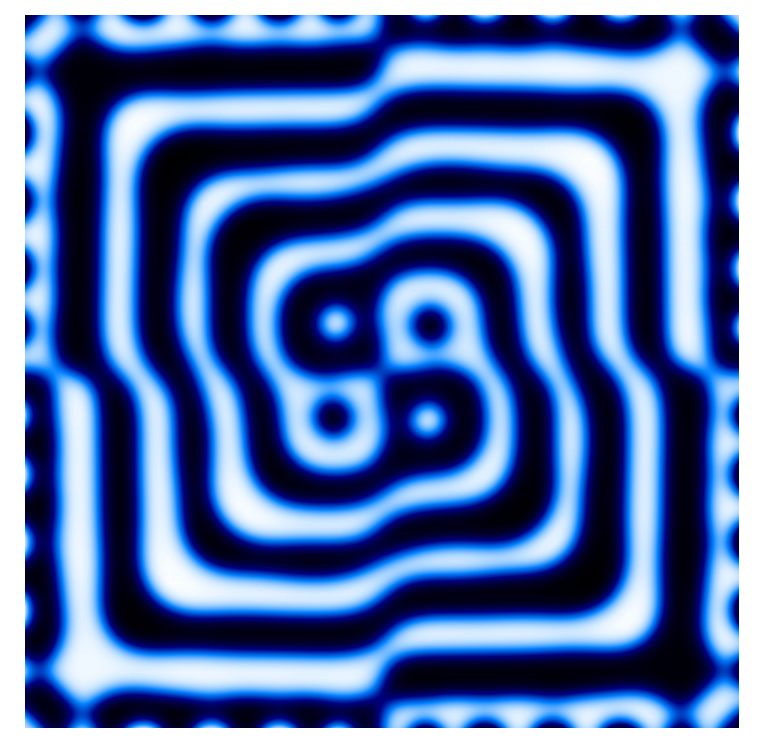}
\label{fig:scenario2:e}
}\hfill
\subfloat[][$t=0.01$]{
\includegraphics[width=\scale]{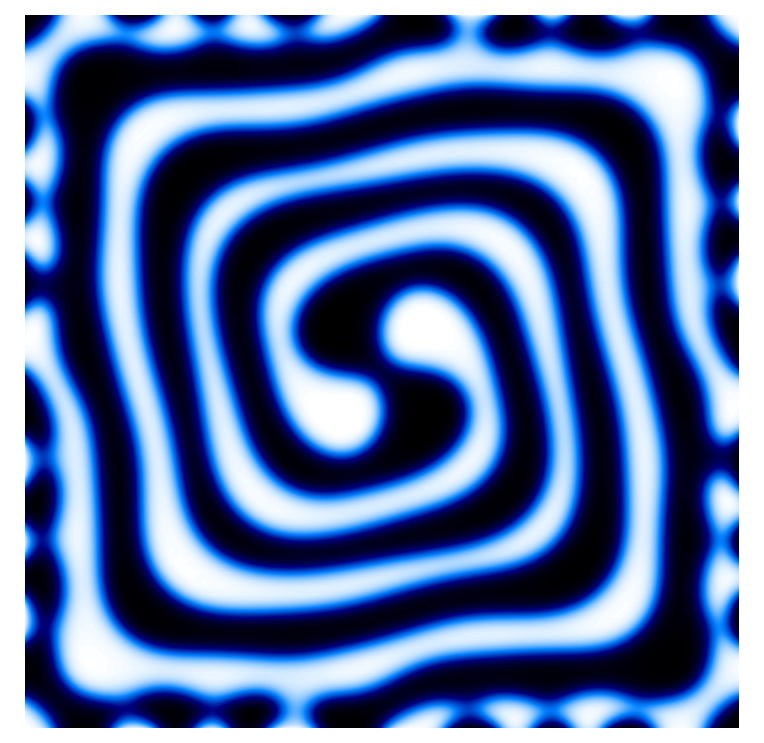}
\label{fig:scenario2:f}
}\hfill
\subfloat[][$t=0.02$]{
\includegraphics[width=\scale]{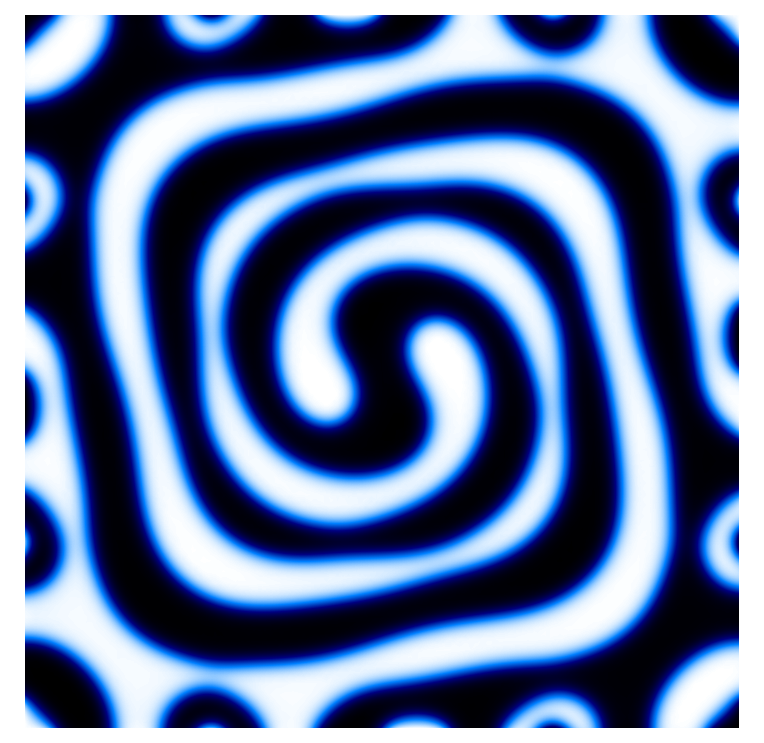}
}\hfill
\subfloat[][$t=0.04$]{
\includegraphics[width=\scale]{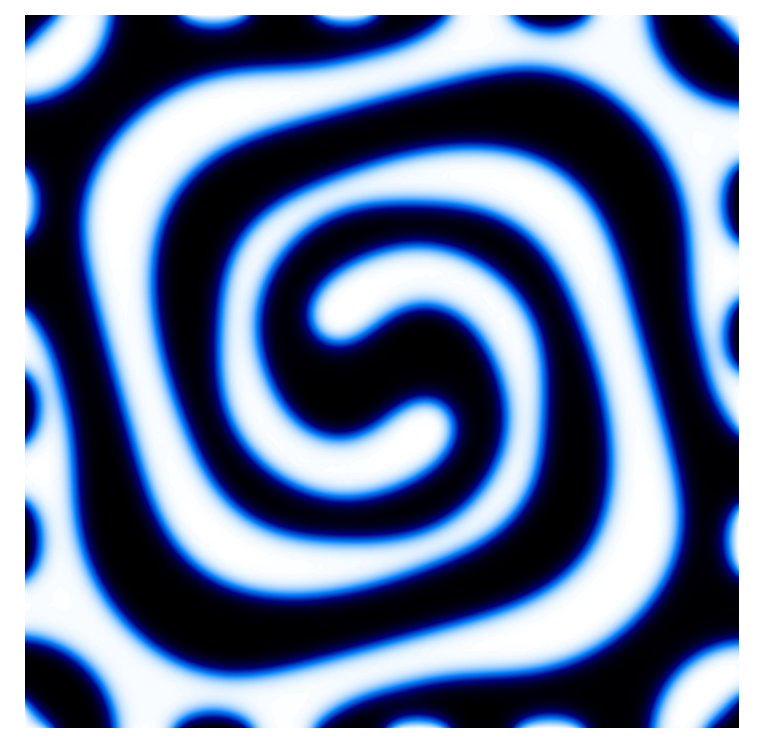}
}\\
\subfloat[][$t=0.08$]{
\includegraphics[width=\scale]{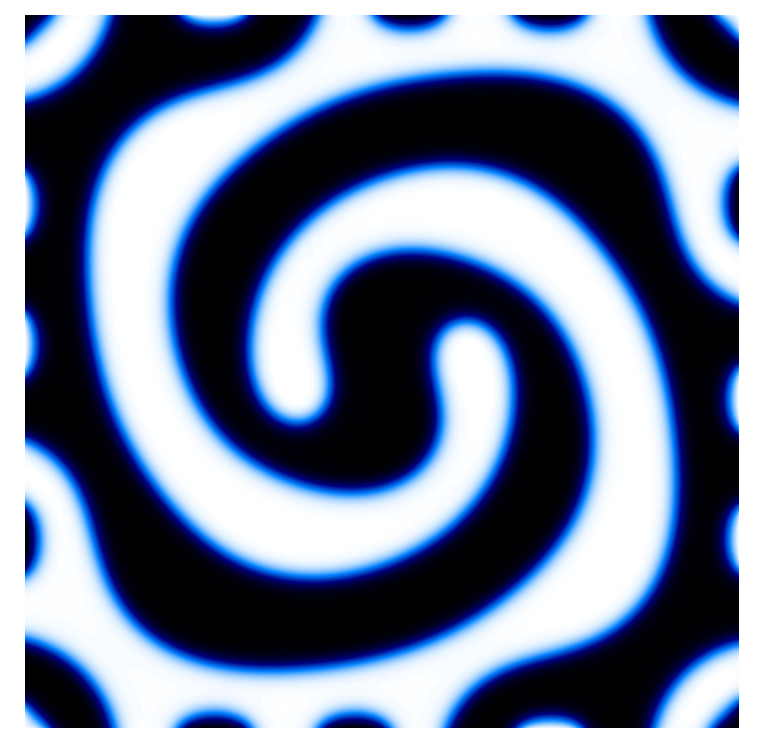}
}\hfill
\subfloat[][$t=0.14$]{
\includegraphics[width=\scale]{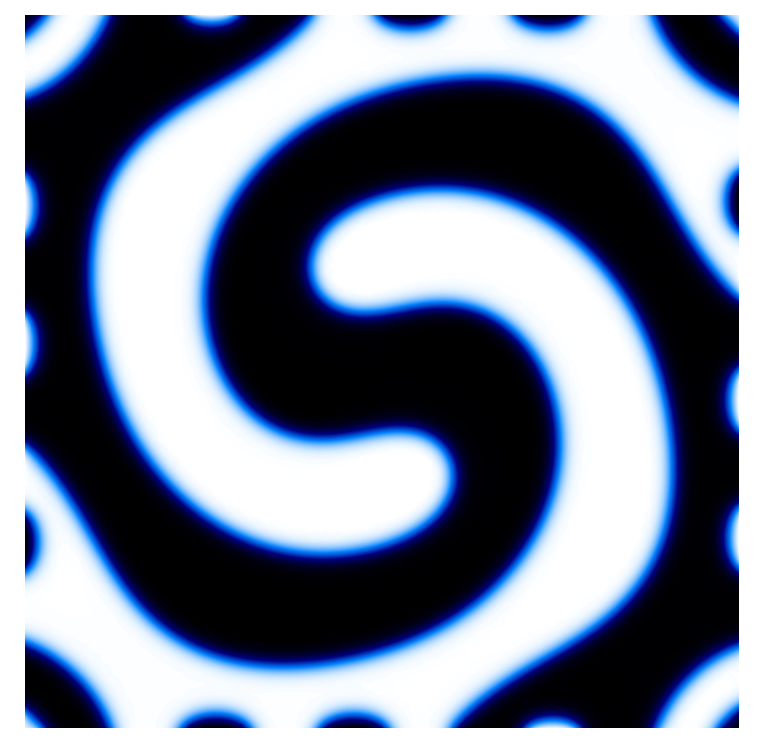}
}\hfill
\subfloat[][$t=0.35$]{
\includegraphics[width=\scale]{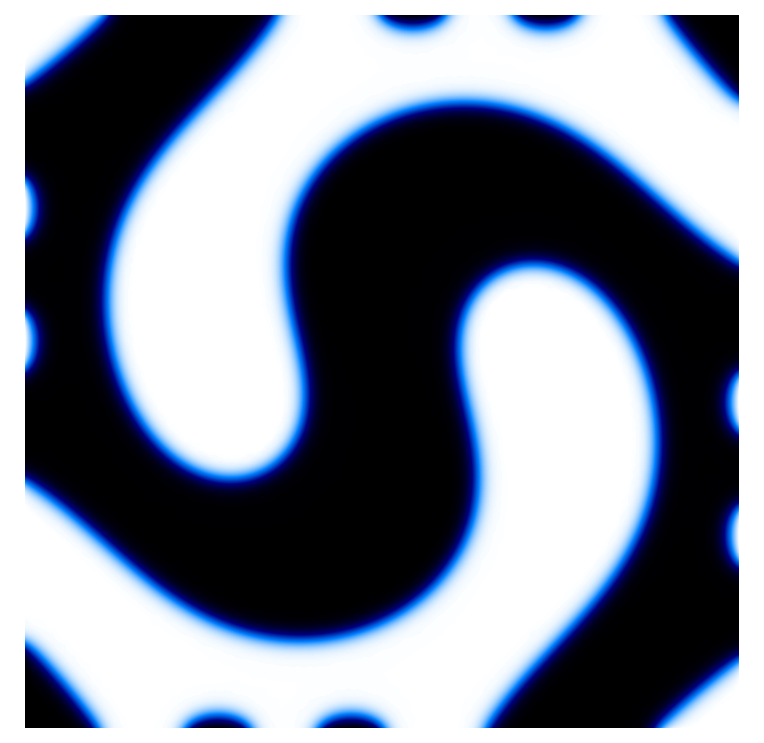}
}\hfill
\subfloat[][$t=1$]{
\includegraphics[width=\scale]{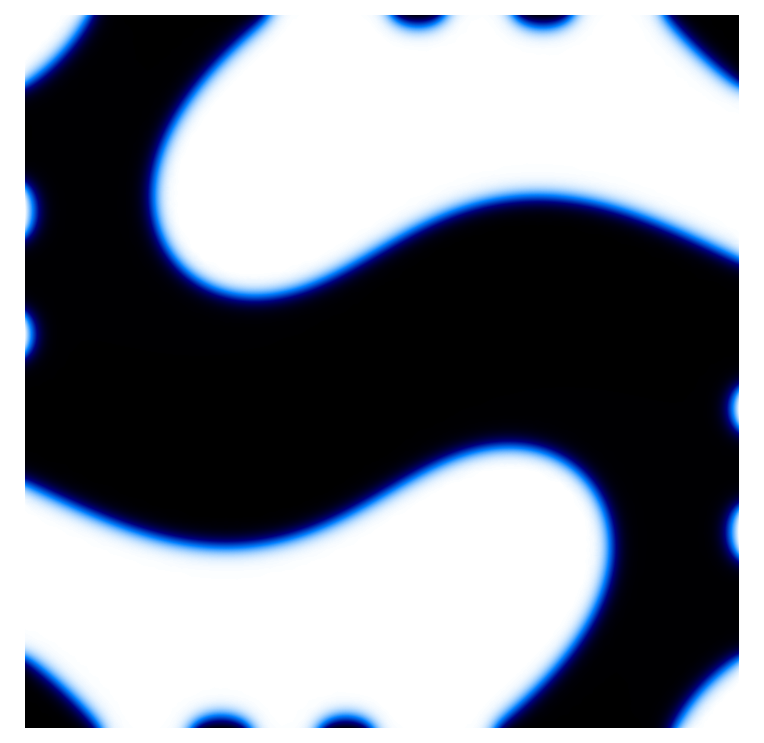}
\label{fig:scenario2:l}
}
\caption{Visualization of the phase separation process considered in Sec.~\ref{subsec:separation}.}
\label{fig:scenario2}
\end{figure}

\begin{figure}
\newcommand{\scale}{.23\textwidth}
\center
\subfloat[][$t=0.0044$]{
\includegraphics[width=\scale]{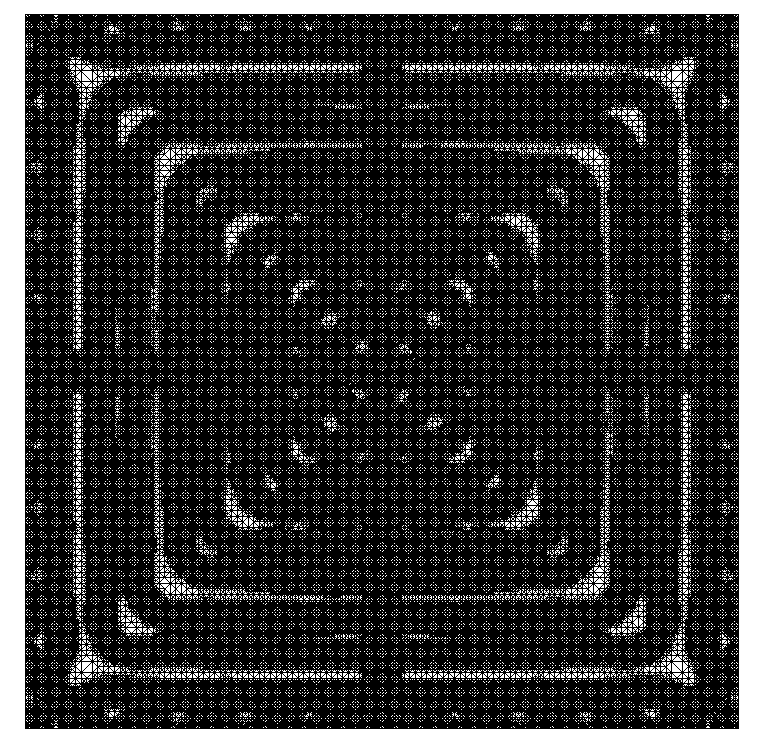}
\label{fig:wire:a}
}\hfill
\subfloat[][$t=0.02$]{
\includegraphics[width=\scale]{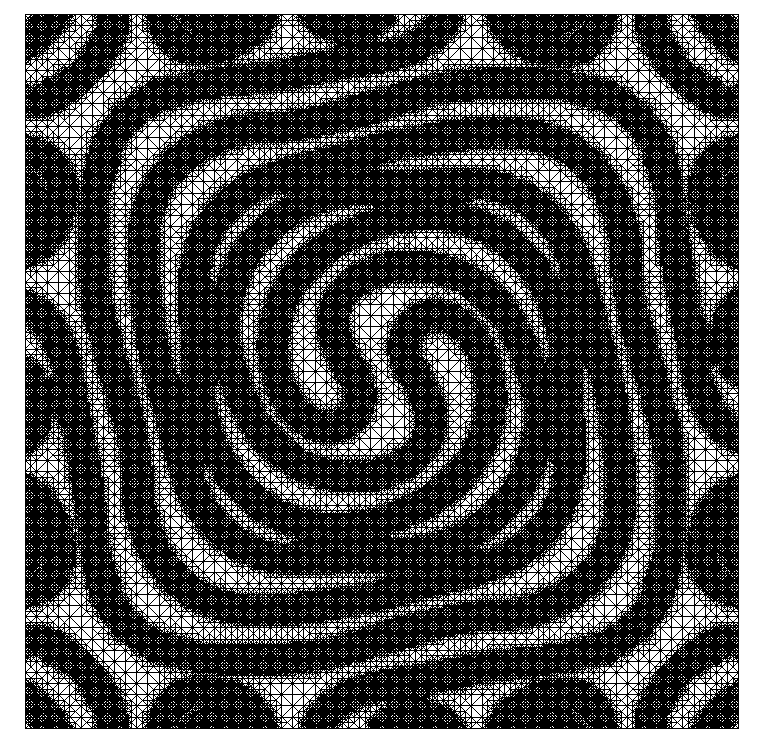}
\label{fig:wire:b}
}\hfill
\subfloat[][$t=0.35$]{
\includegraphics[width=\scale]{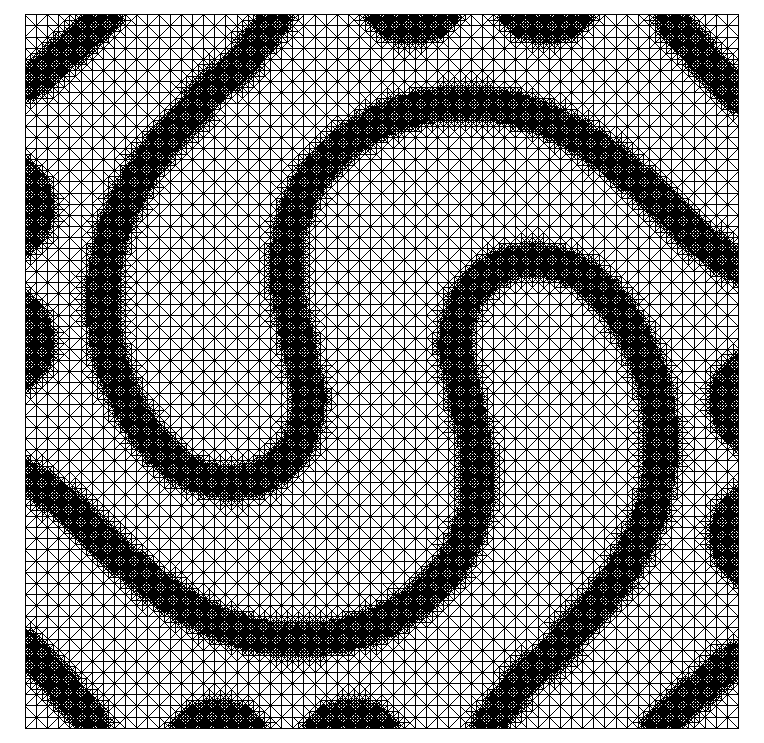}
\label{fig:wire:c}
}
\caption{Adaptive meshes used in Sec.~\ref{subsec:separation}.}
\label{fig:wire}
\end{figure}

\begin{figure}
\newcommand{\scale}{.3\textwidth}
\center
\hfill\subfloat[][Conservation of mass.]{
\includegraphics[width=\scale]{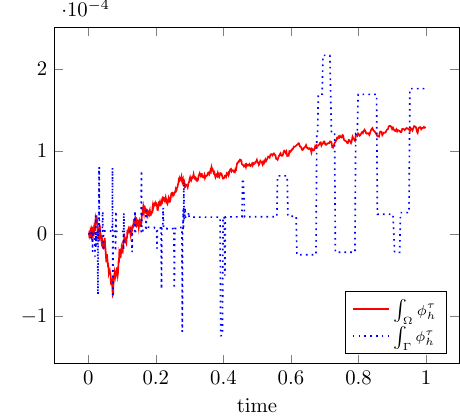}
\label{fig:mass2}
}\hfill
\subfloat[][Decrease of energy.]{
\includegraphics[width=\scale]{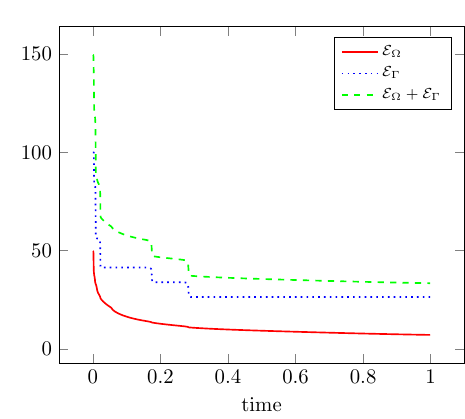}
\label{fig:energy2}
}\hfill~
\caption{Conservation of mass and decrease of total energy.}
\end{figure}

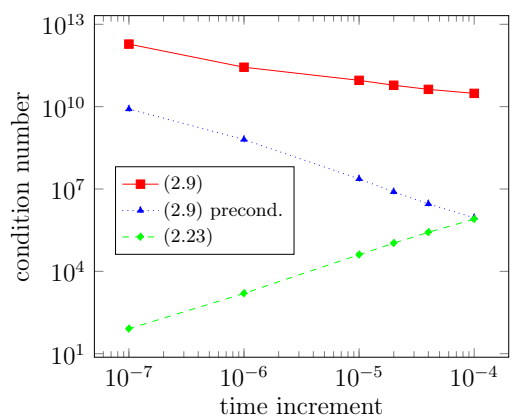
\begin{figure}\center
\newcommand{\scale}{.43\textwidth}
\resizebox{\scale}{!}{
\begin{tikzpicture}
  \begin{axis}[
  title=~,
    xmode=log, ymode=log,
    xlabel={time increment}, ylabel={condition number},
    legend entries={\eqref{eq:feform},\eqref{eq:feform} precond., \eqref{eq:scheme}},
    legend columns=1, legend cell align=left, legend style={font=\scriptsize},
    legend style= {at={(0.05,0.3)},anchor=south west},
    cycle list name={exotic},
    ]
    \addplot[color=red,mark=cube*]   coordinates {
   (0.0001000,30508581842.1604000)
(0.0000400,42668911133.3513000)
(0.0000200,60217171029.2185000)
(0.0000100,91014197140.6907000)
(0.0000010,272930564733.0940000)
(0.0000001,1900248851773.8900000)

    };
    \addplot[dotted,color=blue,mark=triangle*]   coordinates {
  (0.0001000,892992.5340732)
(0.0000400,2869416.0330299)
(0.0000200,7888878.6877471)
(0.0000100,23385916.6174885)
(0.0000010,633895092.8655570)
(0.0000001,8192892492.3760400)

    };
    \addplot[dashed, color=green,mark= diamond*]   coordinates {
    (0.0001000,799011.5533961)
(0.0000400,264855.7333485)
(0.0000200,106292.8528601)
(0.0000100,40387.9036102)
(0.0000010,1572.8376352)
(0.0000001,81.0995932)
    };
   
  \end{axis}
\end{tikzpicture}
}
\caption{Average condition numbers for different time increments.}
\label{fig:cond2}
\end{figure}

\begin{figure}
\center
\hfill\subfloat[][Dimensions of FE-spaces.]{
\includegraphics[width=.43\textwidth]{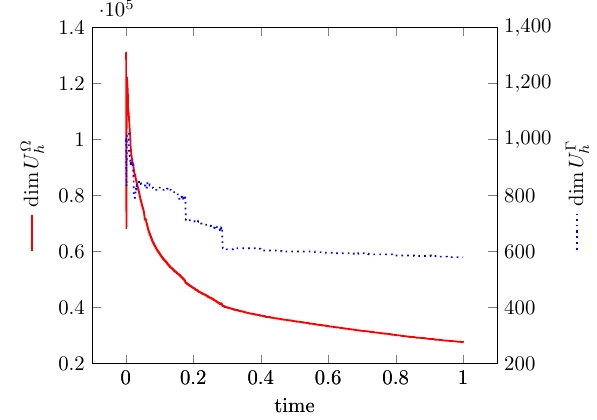}\label{fig:dim}
}\hfill
\subfloat[][Average cg-iterations.]{
\includegraphics[width=.33\textwidth]{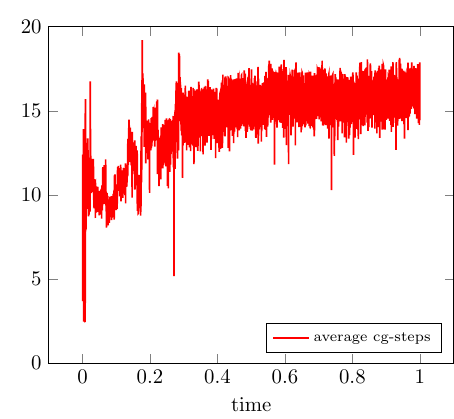}\label{fig:averagecg}
}\hfill~
\caption{Dimensions of FE-spaces and average cg-iterations.}
\end{figure}
\subsection{Scenario 2}\label{subsec:eoc}
In the second scenario, we are interested in the experimental order of convergence (EOC) for our scheme.
Similar to the last section, we choose $W_{\operatorname{pen}}$ (cf. \eqref{eq:penpot}) with penalty parameter $C_{\text{pen}}=250$ for $F$ and $G$ and use $\Omega:=\trkla{0,1}^2$.
The remaining parameters are collected in Tab.~\ref{tab:params1}.
We consider the initial configuration
\begin{align}
\phi\h^0\,:\,\R^2\supset\Omega\rightarrow \tekla{-1,1}&&\trkla{x_1,x_2}\mapsto \Ih{\max\tgkla{0.1\sin\trkla{\pi x_1},0.1\sin\trkla{\pi x_2}}}.
\end{align}
As this configuration is unstable, the two phases will separate and form areas where $\phi$ is close to $\pm1$.
For $h=\sqrt{2}\cdot2^{-7}$ and $\tau=2\cdot10^{-5}$, the evolution of $\phi$ in the time interval $\tekla{0,1}$ is depicted in Fig. \ref{fig:scenario1}.
Thereby, the light color represents  the phase $\phi=+1$ and the dark color represents the phase $\phi=-1$.
In order to visualize the initial condition, we rescaled Fig. \ref{fig:scenario1:initial} so that the dark color corresponds to $\phi=0$ and the light color to $\phi=0.1$.
Again, the separation process reduces the energy of the system (cf.~Fig.~\ref{fig:energy1}).\\
\begin{table}
\center
\begin{tabular}{c|c|c|c|c|c|c}
$\mobO$&$\delta$& $\sigma$& $\mobG$&$\deltaG$& $\kappa$ &$C_\text{pen}$\\
0.01&0.02&2&0.02&0.02&1&250
\end{tabular}
\caption{Parameters used in Section \ref{subsec:eoc}.}
\label{tab:params1}
\end{table}
\begin{figure}
\newcommand{\scale}{.23\textwidth}
\subfloat[][$t=0$ (rescaled)]{
\includegraphics[width=\scale]{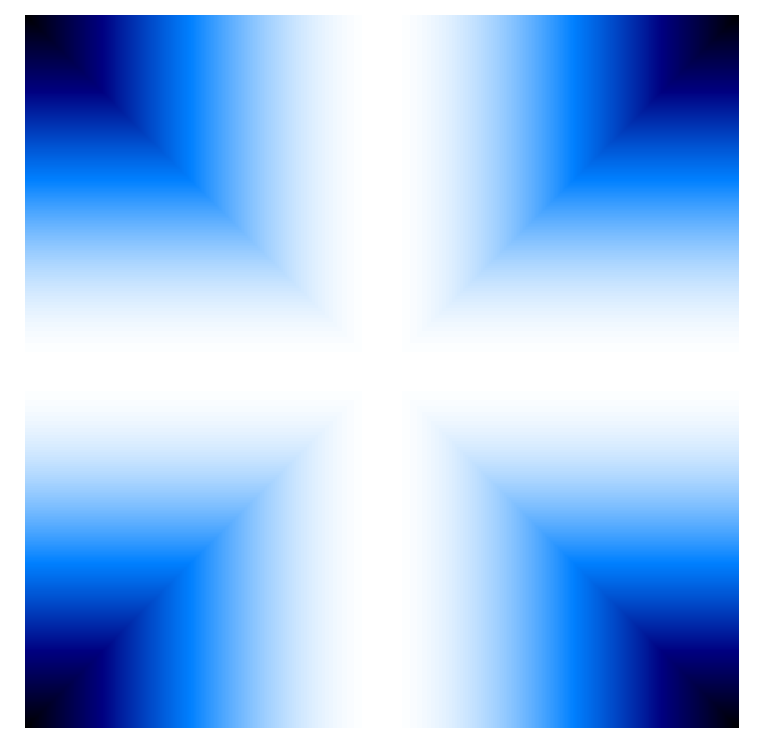}
\label{fig:scenario1:initial}
}\hfill
\subfloat[][$t=0.01$]{
\includegraphics[width=\scale]{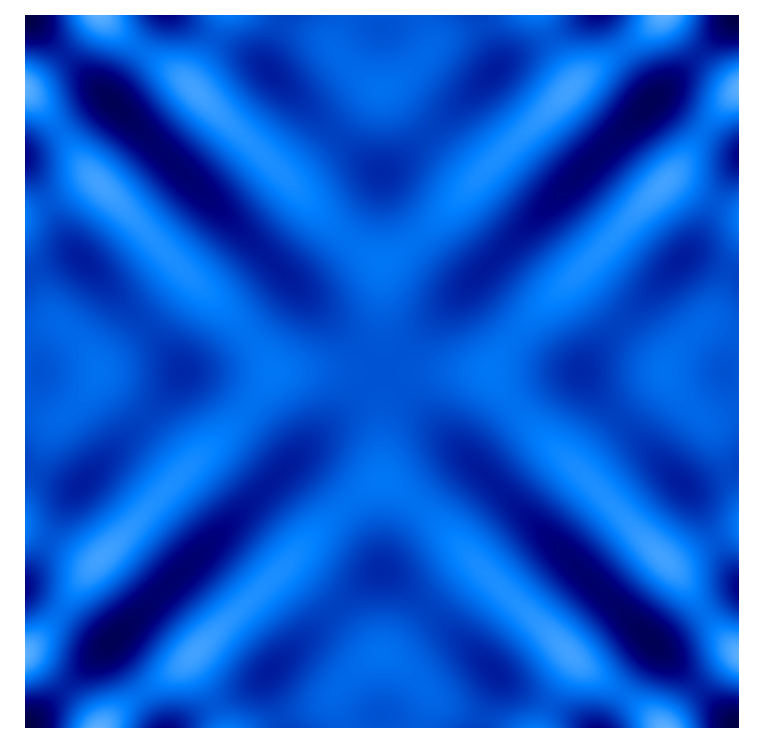}
}\hfill
\subfloat[][$t=0.02$]{
\includegraphics[width=\scale]{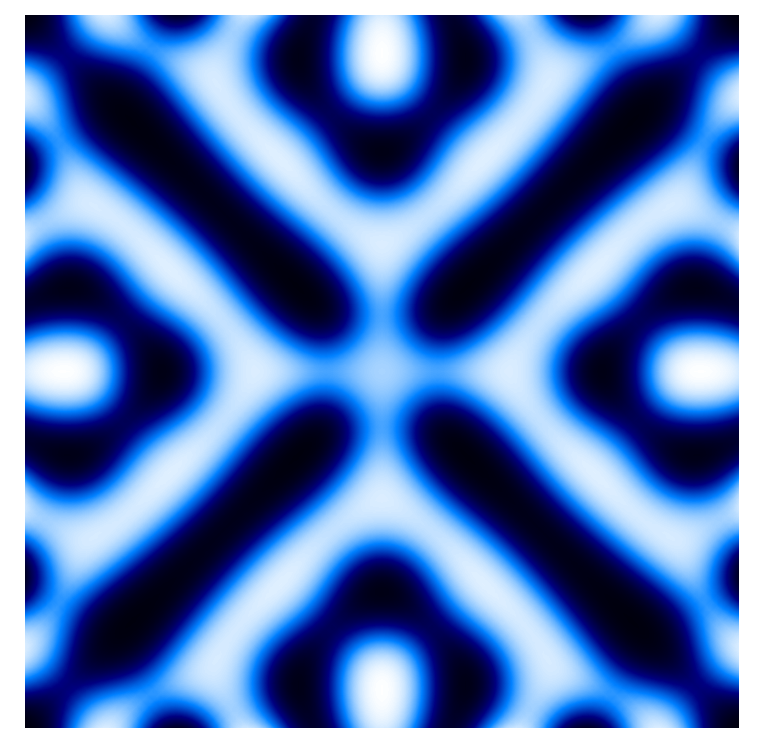}
}\hfill
\subfloat[][$t=0.04$]{
\includegraphics[width=\scale]{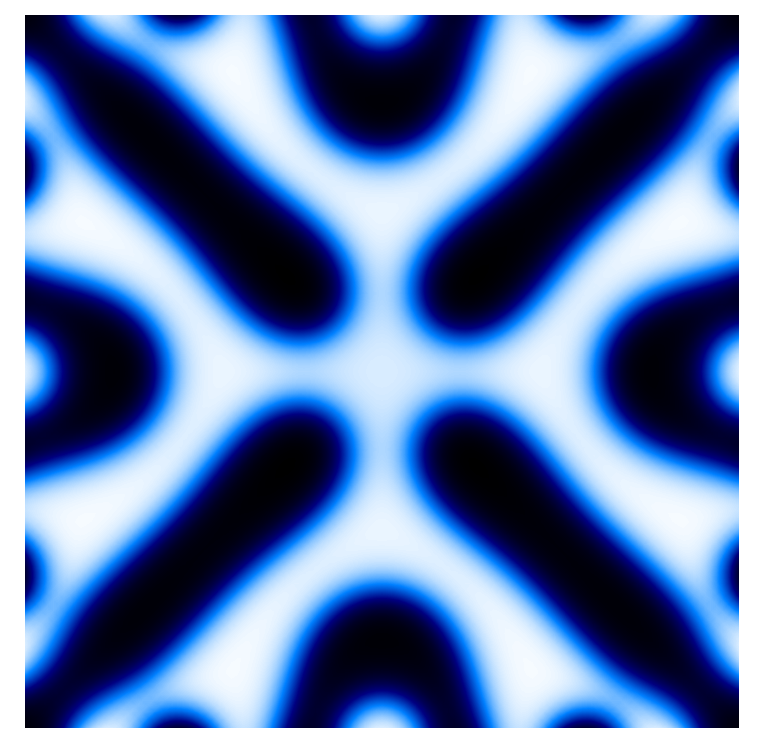}
}\\
\subfloat[][$t=0.1$]{
\includegraphics[width=\scale]{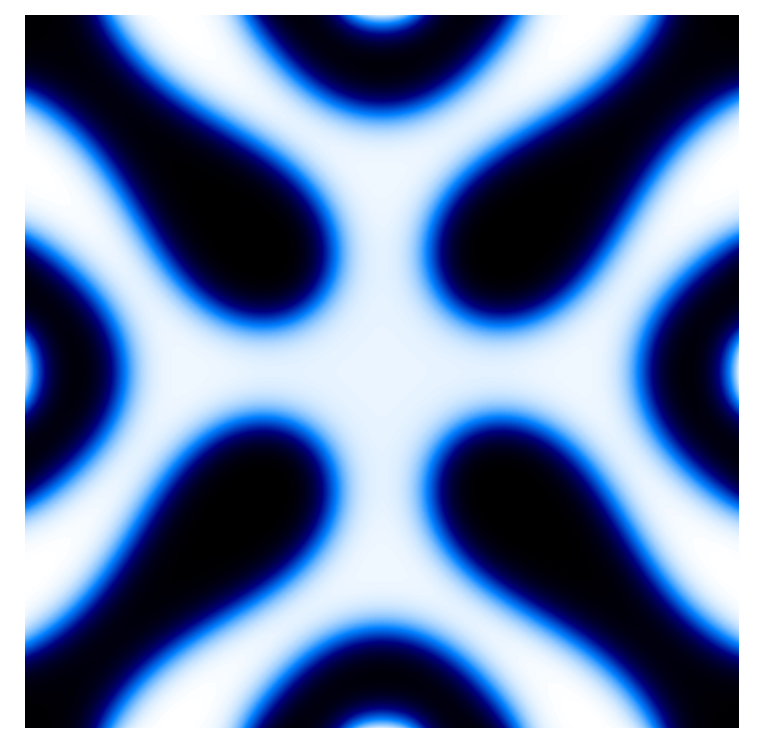}
}\hfill
\subfloat[][$t=0.2$]{
\includegraphics[width=\scale]{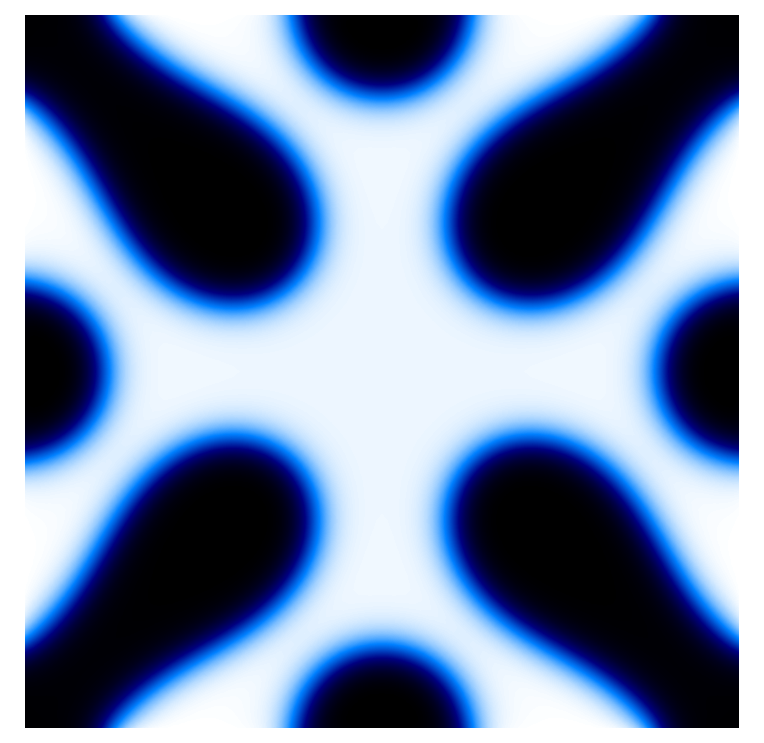}
}\hfill
\subfloat[][$t=0.5$]{
\includegraphics[width=\scale]{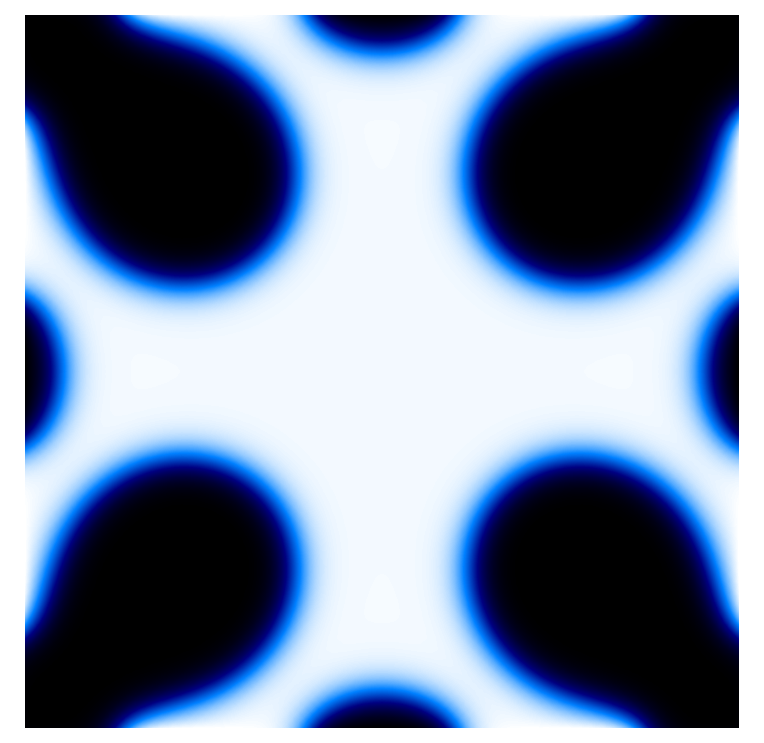}
}\hfill
\subfloat[][$t=1$]{
\includegraphics[width=\scale]{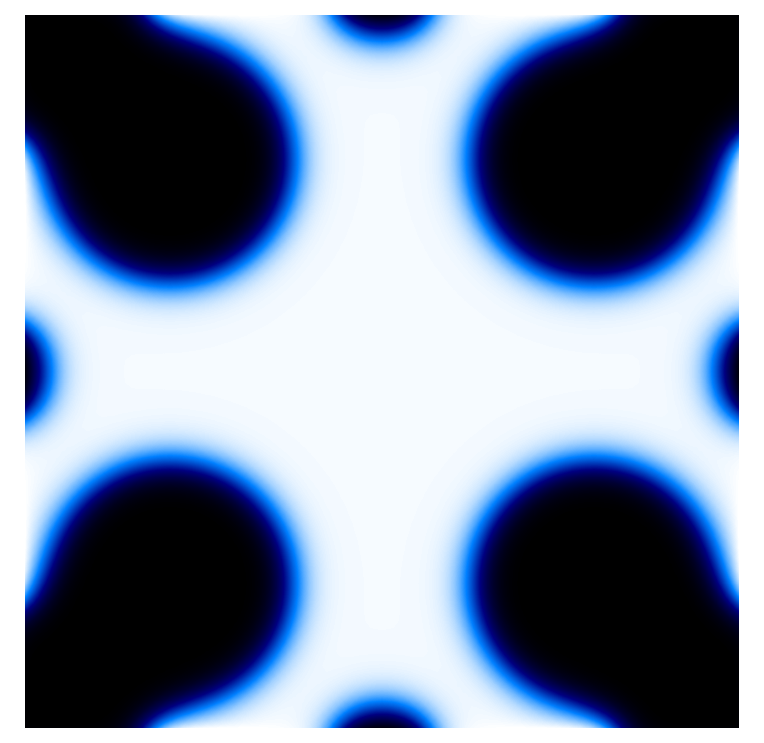}
}
\caption{Visualization of the phase-separation process investigated in Sec.~\ref{subsec:eoc}.}
\label{fig:scenario1}
\end{figure}

To compute an experimental order of convergence, we repeat the above simulation with $\tau= k\cdot 10^{-5}$, $k\in\tgkla{1,2,4}$, and $h=\sqrt{2}\cdot 2^{-l}$, $l\in\tgkla{6,7,8}$.
These spatial discretizations correspond to $\dim\Uhs\in\tgkla{4225,\, 16641,\,66049}$ and $\dim\UhG\in\tgkla{256,\, 512,\,1024}$.
Fixing $\tau=2\cdot10^{-5}$, we use the solution obtained for $h=\sqrt{2}\cdot2^{-8}$ as reference solution $\overline{\phi}$ and define
\begin{align}
\text{err}_h^\Omega:=\norm{\phi\h\tl-\overline{\phi}}_{L^2\trkla{0,1;L^2\trkla{\Omega}}}\,,
\end{align}
where the integration in time is approximated by a trapezoidal rule with step size $2\cdot10^{-4}$.
The experimental order of convergence w.r.t.~$h$ is then defined as
\begin{align}
\text{EOC}_h^\Omega\trkla{i}:={\log\rkla{\tfrac{\text{err}_{h_{i-1}}^\Omega}{\text{err}_{h_{i}}^\Omega}}}\scalebox{1.7}{/}{\log\rkla{\tfrac{h_{i-1}}{h_i}}}\,.
\end{align}
Analogously, we define $\text{err}_h^\Gamma$ and $EOC_h^\Gamma$ using the $L^2\trkla{0,1;L^2\trkla{\Gamma}}$-norm.
As shown in Tab.~\ref{tab:eoc:h}, we obtain order 2.3 for the convergence of $\phi\h\tl$ w.r.t.~$h$ on $\Omega$ and $1.1$ for the convergence of $\trace{\phi\h\tl}$ on $\Gamma$.

\begin{table}
\subfloat[][EOC w.r.t.~$h$ on $\Omega$.]{
\begin{tabular}{c|cc}
$h$ & $\text{err}_h^\Omega$ & $\text{EOC}_h^\Omega$\\\hline
$\sqrt{2}\cdot2^{-6}$ & $3.1\cdot10^{-2}$&-\\
$\sqrt{2}\cdot2^{-7}$ & $6.3\cdot10^{-3}$& 2.3
\end{tabular}}\hfill
\subfloat[][EOC w.r.t.~$h$ on $\Gamma$.]{
\begin{tabular}{c|cc}
$h$ & $\text{err}_h^\Gamma$ & $\text{EOC}_h^\Gamma$\\\hline
$\sqrt{2}\cdot2^{-6}$ & $1.8\cdot10^{-1}$&-\\
$\sqrt{2}\cdot2^{-7}$ & $8.3\cdot10^{-2}$& 1.1
\end{tabular}}
\caption{Experimental order of convergence w.r.t.~$h$.}
\label{tab:eoc:h}
\end{table}
In a similar manner, we fix $h=\sqrt{2}\cdot 2^{-7}$, use the solution corresponding to $\tau=1\cdot10^{-5}$ as reference solution, and define $\text{err}_\tau^\Omega$, $\text{err}_\tau^\Gamma$, $\text{EOC}_\tau^\Omega$, and $\text{EOC}_\tau^\Gamma$ accordingly.
As it can be seen in Tab.~\ref{tab:eoc:tau}, the computed convergence order w.r.t.~the time increment is 1.6 for $\phi\h\tl$ and its trace $\trace{\phi\h\tl}$.
\begin{table}
\subfloat[][EOC w.r.t.~$\tau$ on $\Omega$.]{
\begin{tabular}{c|cc}
$\tau$ & $\text{err}_\tau^\Omega$ & $\text{EOC}_\tau^\Omega$\\\hline
$4\cdot10^{-5}$ & $1.4\cdot10^{-2}$&-\\
$2\cdot10^{-5}$ & $4.5\cdot10^{-3}$& 1.6
\end{tabular}}\hfill
\subfloat[][EOC w.r.t.~$\tau$ on $\Gamma$.]{
\begin{tabular}{l|cc}
$\tau$ & $\text{err}_\tau^\Gamma$ & $\text{EOC}_\tau^\Gamma$\\\hline
$4\cdot10^{-5}$ & $7.3\cdot10^{-2}$&-\\
$2\cdot10^{-5}$ & $2.4\cdot10^{-2}$& 1.6
\end{tabular}}
\caption{Experimental order of convergence w.r.t.~$\tau$.}
\label{tab:eoc:tau}
\end{table}

Again, we conclude this section by evaluating the reliability and the efficiency of our scheme based on the conservation of mass, validity of \eqref{eq:feform:Gamma:phi}, the average condition number, and the average number of cg-iterations needed to solve the auxiliary problem.
As illustrated in Fig.~\ref{fig:mass1}, our scheme conserves $\iOmega\phi\h\tl$ and $\iGamma\phi\h\tl$ perfectly for all considered values of $h$ and $\tau$, if the triangulation is fixed.
As explained in the last section, we may use \eqref{eq:def:muG} to recover $\muGh\nn$ and compute the trace of $\phi\h\nn$ using \eqref{eq:feform:Gamma:phi}.
For $h=\sqrt{2}\cdot2^{-7}$ and $\tau=2\cdot10^{-5}$ the $L^2\trkla{\Gamma}$-norm of the deviation averages to $1.6\cdot10^{-9}$.\\ 
As it was done for the last scenario, we computed the matrices used in the first Newton iteration in the different schemes for various artificial time increments and estimated their condition number using the \textsc{Matlab} function \texttt{condest}.
The average condition numbers based on 51 equidistant points in time are plotted in Fig.~\ref{fig:cond1}.
As before, the matrices arising when solving \eqref{eq:feform} directly are ill-conditioned for small time increments.
A Jacobi preconditioner is able to reduce the condition number drastically, but can not overcome the structural problem of this approach.
Fig.~\ref{fig:cg1} shows the average number of cg-iterations needed to solve the auxiliary problem for $\tau=2\cdot10^{-5}$ and different values of $h$.
Again, the auxiliary problem can be solved with very few iterations. 
\begin{figure}
\newcommand{\scale}{.3\textwidth}
\center
\subfloat[][Decrease of energy.]{
\includegraphics[width=\scale]{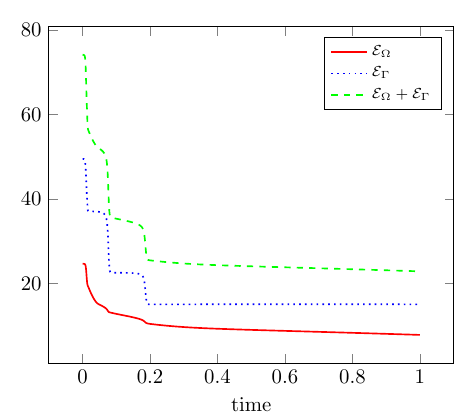}
\label{fig:energy1}
}\hfill
\subfloat[][Conservation of mass. ]{
\includegraphics[width=\scale]{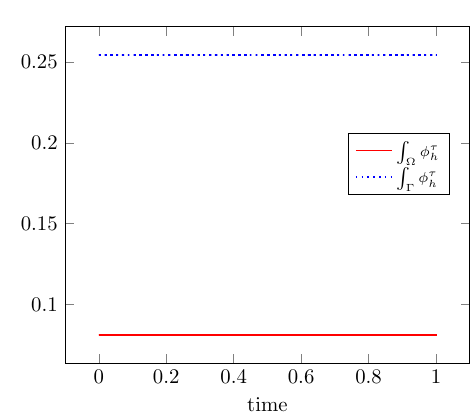}
\label{fig:mass1}
}\hfill
\subfloat[][Average cg-iterations.]{
\includegraphics[width=\scale]{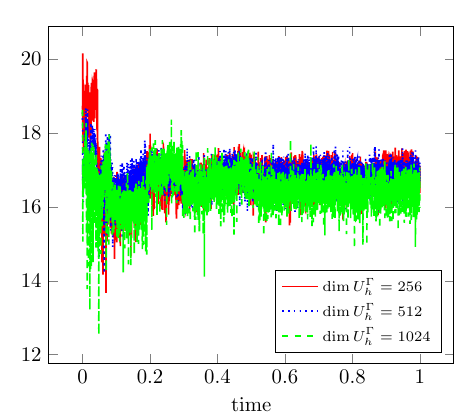}
\label{fig:cg1}
}
\caption{Energy, mass, and average cg-iterations in Scenario 2.}
\end{figure}

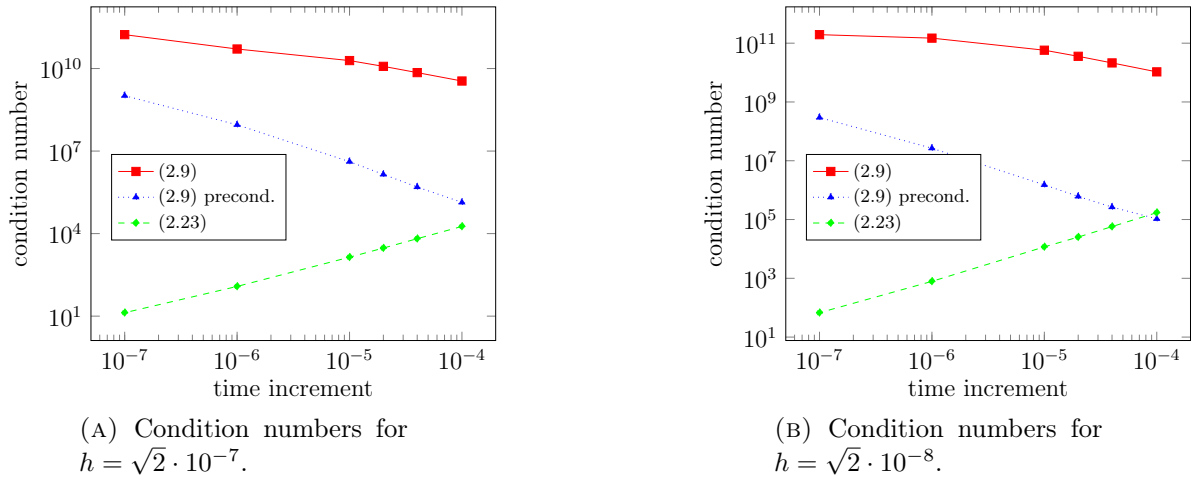
\begin{figure}
\newcommand{\scale}{.42\textwidth}
\center
\subfloat[][Condition numbers for $h=\sqrt{2}\cdot10^{-7}$.]{
\resizebox{\scale}{!}{
\begin{tikzpicture}
  \begin{axis}[
  title=~,
    xmode=log, ymode=log,
    xlabel={time increment}, ylabel={condition number},
    legend entries={\eqref{eq:feform},\eqref{eq:feform} precond., \eqref{eq:scheme}},
    legend columns=1, legend cell align=left, legend style={font=\scriptsize},
    legend style= {at={(0.05,0.3)},anchor=south west},
    cycle list name={exotic},
    ]
    \addplot[color=red,mark=cube*]   coordinates {
    (0.0001000,3518668328.8613400)
(0.0000400,7115961473.8082900)
(0.0000200,12008059726.9612000)
(0.0000100,19473502437.2515000)
(0.0000010,51271098841.5730000)
(0.0000001,170271303274.1200000)
    };
    \addplot[dotted,color=blue,mark=triangle*]   coordinates {
    (0.0001000,136522.3502043)
(0.0000400,493771.9310815)
(0.0000200,1417747.4761820)
(0.0000100,4118852.1038677)
(0.0000010,90065973.4875475)
(0.0000001,1028732059.7508800)

    };
    \addplot[dashed, color=green,mark= diamond*]   coordinates {
    (0.0001000,18246.6621711)
(0.0000400,6524.7075521)
(0.0000200,2981.3324872)
(0.0000100,1387.1965894)
(0.0000010,120.1754684)
(0.0000001,13.3060032)

    };
   
  \end{axis}
\end{tikzpicture}
}}
\hfill
\subfloat[][Condition numbers for $h=\sqrt{2}\cdot10^{-8}$.]{
\resizebox{\scale}{!}{
\begin{tikzpicture}
  \begin{axis}[
  title=~,
    xmode=log, ymode=log,
    xlabel={time increment}, ylabel={condition number},
    legend entries={\eqref{eq:feform},\eqref{eq:feform} precond., \eqref{eq:scheme}},
    legend columns=1, legend cell align=left, legend style={font=\scriptsize},
    legend style= {at={(0.05,0.3)},anchor=south west},
    cycle list name={exotic},
    ]
    \addplot[color=red,mark=cube*]
    coordinates {
    (0.0001000,10562723525.2308000)
(0.0000400,21296860264.8004000)
(0.0000200,35794004930.3980000)
(0.0000100,57701686209.2617000)
(0.0000010,148478309196.6580000)
(0.0000001,194829221264.5930000)
    };
    \addplot[dotted, color=blue,mark=triangle*] coordinates {
       (0.0001000,105528.4982702)
(0.0000400,265122.4787326)
(0.0000200,611449.7673469)
(0.0000100,1510936.4638895)
(0.0000010,26677162.5457058)
(0.0000001,295451855.2050900)
 
    };
    \addplot[dashed, color=green,mark= diamond*]   coordinates {
    (0.0001000,171634.4115705)
(0.0000400,57477.0181449)
(0.0000200,25262.3092157)
(0.0000100,11722.5037508)
(0.0000010,786.9799959)
(0.0000001,67.4665984)

    };
   
  \end{axis}
\end{tikzpicture}
}}
\caption{Average condition numbers for different time increments.}
\label{fig:cond1}
\end{figure}

\begin{appendix}
\section{Appendix}
For the reader's convenience, we provide the generalized Poincar\'e inequality which can be found in \cite{Alt2016}.
\begin{lemma}\label{lem:poincare}
Let $\Omega\subset\R^d$ be open, bounded and connected with Lipschitz boundary $\partial\Omega$.
Moreover, let $1<p<\infty$ and let $\mathcal{M}\subset W^{1,p}\trkla{\Omega}$ be nonempty, closed and convex.
Then the following items are equivalent for every $u_0\in\mathcal{M}$:
\begin{enumerate}
\item There exists a constant $C_0<\infty$ such that for all $\xi\in\R$
\begin{align*}
u_0+\xi\in\mathcal{M}&&\Longrightarrow&&\abs{\xi}\leq C_0\,.
\end{align*}
\item There exists a constant $C<\infty$ with
\begin{align*}
\norm{u}_{L^p\trkla{\Omega}}\leq C\trkla{1+\norm{\nabla u}_{L^p\trkla{\Omega}}}&&\text{for~all~} u\in\mathcal{M}\,.
\end{align*}
\end{enumerate}
\end{lemma}
\end{appendix}

\bibliographystyle{amsplain}

\end{document}